\documentclass[11pt]{amsart}
\usepackage{pstricks}
\usepackage{color}
\theoremstyle{plain}
\newtheorem{thm}{Theorem}[section]
\newtheorem{theorem}[thm]{Theorem}

\newtheorem{lemma}[thm]{Lemma}
\newtheorem{corollary}[thm]{Corollary}
\newtheorem{proposition}[thm]{Proposition}
\newtheorem{problem}[thm]{Problem}
\theoremstyle{definition}
\newtheorem{remark}[thm]{Remark}

\newtheorem{definition}[thm]{Definition}
\newtheorem{claim}[thm]{Claim}

\newtheorem{question}[thm]{Question}

\numberwithin{equation}{section}
\newcommand{\sO}{{\mathcal O}}

\newcommand{\C}{{\mathbb C}}

\newcommand{\Q}{{\mathbb Q}}
\newcommand{\R}{{\mathbb R}}

\newcommand{\Z}{{\mathbb Z}}
\newcommand{\PP}{\ensuremath{\mathbb{P}}}

\title [Homeomorphic 4-folds]{On the homeomorphism type of smooth projective fourfolds}

\author{Keiji Oguiso}
\author{Thomas Peternell}

\address{Mathematical Sciences, the University of Tokyo, Meguro Komaba 3-8-1, Tokyo, Japan and Korea Institute for Advanced Study, Hoegiro 87, Seoul, 
133-722, Korea}
\email{oguiso@ms.u-tokyo.ac.jp}

\address{Mathematisches Institut, Universit\"at Bayreuth, 95440 Bayreuth, Germany}
\email{thomas.peternell@uni-bayreuth.de}

\thanks{The first named author is supported by the ERC 2013 Advanced Research Grant - 340258 - TADMICAMT, JSPS Grant-in-Aid (S) No 25220701, JSPS Grant-in-Aid (S) 15H05738, JSPS Grant-in-Aid (B) 15H03611, and KIAS Scholar Program. }


\begin{document}

\maketitle
\tableofcontents

\begin{abstract}   In this paper we study smooth complex projective $4$-folds which are topologically equivalent. 
First we show that Fano fourfolds are never oriented homeomorphic to Ricci-flat projective fourfolds and that Calabi-Yau manifolds and hyperk\"ahler manifolds in dimension $\ge 4$ are never oriented homeomorphic. Finally, we give a coarse classification of  
smooth projective fourfolds which are oriented homeomorphic to a hyperk\"ahler fourfold which is deformation equivalent to the Hilbert scheme $S^{[2]}$ of two points of a projective K3 surface $S.$ 
\end{abstract}

\section{Introduction }

Throughout this note, we work in the category of complex projective manifolds. A projective manifold of complex dimension $n$  will shortly be called an $n$-fold.

Whenever we speak of the topology of projective manifolds $X$ and $Y,$ we use the Euclidean topology of $X$ and $Y$ equipped with the  natural orientation coming from the complex structure.

\begin{definition} 
The manifolds $X$ and $Y$ are called {\it o-homeomorphic} if there is an {\it orientation preserving} homeomorphism $\varphi : X \to Y.$  Note then that $\dim\, X = \dim\, Y.$
\end{definition} 

 Similarly one defines the notions of o-diffeomorphy  (of class $\mathcal C^{\infty}$ always in this note) and o-homotopy. By definition, o-diffeomorphy  implies o-homeomorphy
  and o-homeomorphy implies o-homotopy. We note that the intersection numbers are preserved under any o-homeomorphisms (but not necessarily under homeomorphisms);
 this is the main reason to consider o-homeomorphisms rather than homeomorphisms. 

 The main question we consider is the following natural

 \begin{problem} \label{mainproblem}  Let $X$ be a complex projective manifold. Describe all smooth complex projective structure on the underlying oriented topological manifold  of $X,$ i.e. study complex projective
 manifolds $Y$ o-homeomorphic to $X.$ 
 
 \end{problem} 
 
 More generally, one may ask for all {\it complex structures} on $X.$ However we will not discuss non-projective complex structures in this paper at all.

 The answer is of course well-known  in dimension $n = 1$; $X$ and $Y$ are o-homeomorphic if and only if the genus $h^1(\sO_X)$ and $h^1(\sO_Y)$ are the same. In dimension $n \ge 2,$ this problem has a 
 long history, at least since Hirzebruch and Kodaira \cite{HK57}. However, the complete answer, even under the assumption that $Y$ is projective, seems to be known only when $X = \PP^n,$ the projective space of complex dimension $n$ (Hirzebruch-Kodaira \cite{HK57}, Yau \cite{Ya77}), an odd dimensional smooth quadric hypersurface (Brieskorn, \cite{Br64}), abelian varieties (Catanese, \cite{Ca04}) and simply connected smooth surfaces (Friedman \cite{Fr82}) by now. Remarkably, Libgober-Woods \cite{LW90} showed that a compact K\"ahler manifold of dimension at most $6$ which is homotopy equivalent to {\rm $\mathbb P^n$} is biholomorphically equivalent
 to {\rm $\mathbb P^n$}. There are many important results for varieties with large fundamental group, e.g., by Mostow, Siu, Jost-Yau, Catanese, Bauer-Catanese. We refer to the excellent survey paper by to Catanese \cite{Ca15} for any details.

The case of simply connected threefolds seems already very difficult except above mentioned cases (see eg. \cite{Na96},  \cite{CP94}, \cite{Kol91} and references therein).  Indeed, this note is much inspired by the following very interesting but highly difficult questions asked by I. Nakamura \cite[Page 538]{Na96} and the second named author, \cite{CP94}, which are completely open even now:
\begin{question}\label{q1}
\begin{enumerate}
\item Is there a Calabi-Yau 3-fold $Y$ which is o-homeomorphic to a smooth cubic $3$-fold $X = (3)  \subset \PP^4$?  
\item Is there a Calabi-Yau 3-fold $Y$ which is o-homeomorphic to $X = (2) \cap (2)  \subset \PP^5,$ a smooth complete intersection of  two smooth quadratic hypersurfaces in $\PP^5$?  
\end{enumerate}
\end{question}

Notice that in dimension at least three, the Kodaira dimension is not a diffeomorphism invariant of compact K\"ahler manifolds. This was first observed by Catanese-LeBrun \cite{CL97}, and extended by
 Rasdeaconu \cite{Ra06}.  
In the present paper we address this circle of problems and questions {\it in dimension $4$} and higher; see Theorems \ref{cor:thm1}, \ref{xxx},  \ref{thm2}, \ref{thm3}, \ref{thm4}, \ref{thm5}, \ref{thm1}.
In some aspects, the problem is more tractable for $4$-folds than $3$-folds, as the middle Betti cohomology group $(H^4(X, \Z), (*,**)_X)$ with intersection form $(*, **)_X$ has a rich structure encoding non-trivial algebro-geometric and lattice theoretic informations, and not only the first Pontrjagin class, but also the second Pontrjagin class gives some non-trivial constraints on Chern classes. Moreover, the Riemann-Roch formula for the holomorphic Euler characteristic on a $4$-fold $X$ includes the topological term $c_4(X),$ whereas in dimension $3,$ the Chern class $c_3(X) $ does not appear in Riemann-Roch.

Recall a smooth projective $n$-fold $X$ with ample anti-canonical class $-K_X$ is said to be a {\it Fano $n$-fold}. 
We denote by $r_X$ the Fano index of $X,$ i.e., the largest integer $r$ such that the canonical bundle $K_X$ is divisible by $r$ in the 
Picard group ${\rm Pic}(X).$ 

We call a smooth projective $n$-fold $X$ a {\it Calabi-Yau $n$-fold} (resp. a {\it hyperk\"ahler $n$-fold}) if $\pi_1(X) = \{1\},$ $H^0(\Omega_X^{k}) = 0$ for all integers $k$ such that $0 < k < n$ and $H^{0}(X, \Omega_X^n) = \C\omega_X$ with nowhere vanishing holomorphic $n$-form $\omega_X$ (resp. if $\pi_1(X) = \{1\},$ $H^0(\Omega_X^{2}) = \C \eta_X$ with everywhere non-degenerate holomorphic $2$-form $\eta_X$). Almost by definition, hyperk\"ahler manifolds are of even dimension.

Our first main result is the following

\begin{theorem} \label{cor:thm1} 
Let $X$ be Fano manifold of even dimension $n \geq 4.$ 
 If $n \geq 6,$ assume additionally that $b_2(X) = 1.$ Then $X$ cannot be o-homeomorphic to a compact K\"ahler manifold $Y$ with 
$c_1(Y) = 0.$ In particular, Fano $4$-folds are never o-homeomorphic to a Calabi-Yau  $4$-fold or a hyperk\"ahler $4-$fold.
\end{theorem} 

Theorem \ref{cor:thm1}  will be a consequence of a more general, result, Theorem \ref{thm1}, see Section \ref{s2}. 
 It is clear that Fano $n$-folds  ($n \ge 1$) are not homeomorphic to an abelian variety, as  a Fano $n$-fold is simply connected.

 We also remark that Catanese-LeBrun \cite{CL97} have shown that for every even integer $n \geq 4$ there is a Fano manifold $X$ of dimension $n$ which is o-diffeomorphic to a projective manifold $Y$  of general type; see Remark \ref{qu3} below.

Note that a smooth cubic $4$-fold is Fano  with second Betti number $b_2 = 1.$ Cubic $4$-folds have attracted much attention these days, especially in connection with the rationality problem (\cite{Ku10}) and its miraculous relations with hyperk\"ahler $4$-folds due to an observation by Beauville-Donagi (\cite{BD85}). So, it is certainly of interest to take a closer look at a smooth cubic $4$-folds also from a topological point of view, or, more generally to del Pezzo  $4$-folds. Recall that a del Pezzo $n$-fold is a Fano 
$n$-fold $X$ such that $r_X = n-1.$ 

In this direction, we first obtain the following, actually in all dimensions

\begin{theorem} \label{xxx} 
Let $X$ and $Y$ be o-homeomorphic Fano manifolds of dimension $n$ with $b_2(X) = b_2(Y)  = 1.$  Then 
\begin{enumerate}
\item $r_X = r_Y$ and $c_1(X)^n = c_1(Y)^n.$
\item If in addition $X$ is a del Pezzo manifold, then $X$ and $Y$ are deformation equivalent. 
\end{enumerate} 
\end{theorem} 

In the case of cubic $4$-folds, more is true. 

\begin{theorem} \label{thm2} Assume that a smooth projective $4$-fold $Y$ is o-homeomorphic to a smooth cubic $4$-fold $X.$ Then 
\begin{enumerate}
\item Either $Y$ is deformation equivalent to $X$ or $Y$ is a smooth projective $4$-fold with ample $K_Y$ and with the same Hodge numbers as a smooth cubic $4$-fold. 
\item Suppose that $Y$ is log o-homeomorphic to $X.$ Then $Y$ is isomorphic to a smooth cubic $4$-fold. 
\end{enumerate}
\end{theorem}

The second part of Theorem \ref{thm2} uses the new notion of log o-heomeomorphy; the relevant definition is

\begin{definition} Let $X$ and $Y$ be smooth projective $n$-folds. $X$ and $Y$ are said to be  {\it log o-homeomorphic}, if there exists an o-homeomorphism $\varphi: X \to Y$ and smooth divisors $D_X \subset  X,$
$D_Y \subset Y$ such that $\varphi(D_X) = D_Y.$ We then also say that the pairs $(X,D_X)$ and $ (Y,D_Y)$ are o-homeomorphic. 

\end{definition} 

\begin{remark}\label{qu3} 

It may be quite difficult but certainly interesting to see whether there really exists an o-homeomorphism $\varphi : X \to Y$ with ample $K_Y$  from a Fano manifold  $X$ with $b_2(X) = 1.$ Theorem \ref{thm2} (2) shows that in this case $\varphi$ does not preserve any smooth divisor  if $X$ is a cubic $4$-fold. 

\end{remark}

Of course, we may ask whether Theorem \ref{thm2} holds for other Fano  $4$-folds,  e.g., other del Pezzo  $4$-folds. We treat the case of del Pezzo 4-folds of degree $5$ in Section \ref{s3a}. This section provides
also a result in any dimension.

By definition, both a Calabi-Yau $2$-fold  and a hyperk\"ahler $2$-fold are K3 surfaces and all K3 surfaces are deformation equivalent ( \cite[Theorem 13]{Ko64}), in particular, o-homeomorphic. Calabi-Yau manifolds and hyperk\"ahler manifolds are close in some sense. However, concerning topological structure, it turns out that they are different  in dimension greater than $2$: 

\begin{theorem}\label{thm3} Let $n\geq 2$ be an integer. 
\begin{enumerate} 
\item  A Calabi-Yau $2n$-fold and a hyperk\"ahler $2n$-fold can never be o-homeomorphic.
\item Assume $n$ is an even integer. Let $Y$ be a hyperk\"ahler  $2n$-fold and let $X$ be a compact K\"ahler manifold with $c_1(X) = 0$ o-homeomorphic to $Y.$ Then $X$ is again hyperk\"ahler.
\end{enumerate} 
\end{theorem}

Our proof is extremely simple, but as far as the authors are aware, this is not noticed before. We prove Theorem \ref{thm3} in Section \ref{s4}. 

In the following we will need the notion of the numerical Kodaira dimension. Recall that the numerical Kodaira dimension $\nu(Y) := \nu(Y, K_Y)$ of a minimal projective $n$-fold $Y,$ hence $K_Y$ is nef, is defined by
$$\nu(Y) := \nu(Y, K_Y) := {\rm max}\, \{ k\,\, |\,\,  K_Y^k \not= 0 \in H^{2k}(Y, \Q)\,\, \}\,\, .$$
This notion is introduced by Kawamata \cite{Ka85}. In the same paper, he also shows that  $\nu(Y) \ge \kappa(Y)$  and the equality holds if and only if $K_Y$ is semi-ample. 

%The next two theorems  (Theorems \ref{thm4}, \ref{thm5}) are the heart of this note. 
Let $S$ be a projective K3 surface. The Hilbert scheme $S^{[n]}$ of  $0$-dimensional closed subschemes of length $n$ of $S$ and their deformation $X$  are hyperk\"ahler $2n$-folds, even though $X$ is not projective in general. (\cite{Fu83}, \cite{Be83}, see also \cite[Section 21.2]{GHJ03}).

\begin{theorem}\label{thm4} Let $X$ be a  projective hyperk\"ahler  $4$-fold which is deformation equivalent to $S^{[2]}$ of a projective K3 surface $S.$

Let $Y$ be a smooth projective $4$-fold. Assume that $Y$ is o-homeomorphic to $X.$ Then: 

\begin{enumerate}

\item $h^0(Y, \sO_Y(K_Y)) = 1.$ In particular, $Y$ is not of Kodaira dimension $\kappa(Y) = -\infty.$ 

\item $K_Y$ is nef and  $\nu(Y)$ is either $0$ or $2.$  In particular, $Y$ is not of general type. 

\item If $\nu(Y) = 0,$ then $Y$ is a hyperk\"ahler $4$-fold. 

%\item There is a smooth projective $4$-fold $Y$ with  $\kappa(Y) = \nu(Y) = 2$ which is o-diffeomorphic to $X.$ 

\end{enumerate}
\end{theorem}
 
There are only two known classes of hyperk\"ahler $4$-folds, up to deformation equivalence. The other one are generalized Kummer $4$-folds and their deformations. Members  of these two
classes  cannot be homeomorphic, as the second Betti numbers are different, namely $7$ and $23$ (\cite{Be83}, see also \cite[Section 24.4]{GHJ03}). We prove Theorem \ref{thm4}  in Section \ref{s5}; however
we do not know whether varieties $Y$ with $\nu(Y) = 2$ in Theorem \ref{thm4}(2) really exist. 
We end the paper with some speculations on exotic structures on hyperk\"ahler manifolds, see Question \ref{que}.

{\bf Acknowledgement.} Most thanks of the first named author go to Fabrizio Catanese  for his invitation to Universit\"at Bayreuth with full financial support from the ERC 2013 Advanced Research Grant - 340258 - TADMICAMT. The main part of this work  started there. Very special thanks go to Fr\'ed\'eric Campana whose comments and remarks greatly helped to improve the paper. In particular this concerns Theorems \ref{thm1} and \ref{xxx}.
 We  would like to express our thanks to Fabrizio Catanese, Yujiro Kawamata, Yongnam Lee, Bong Lian and Kiwamu Watanabe for fruitful discussions. 
%Last but not least 
%we thank the referee for his very valuable comments. 

\section{Preliminaries from differential topology} \label{sprep}

First we recall basic definitions and properties of characteristic classes of projective manifolds, assuming the notion of Chern classes of projective manifolds.

Let $X$ be a projective manifold. Let $c_i(E) \in H^{2i}(X, \Z)$ be the $i$-th Chern class of a topological $\C$-vector bundle $E$ on $X.$ We denote by $c_i(X) := c_i(T_X)$ the $i$-th Chern class of $X.$ Here $T_X$ is the holomorphic tangent bundle of $X.$ We denote by $c_{*}(E)$ the total Chern class $\sum_{i \ge 0} c_i(E)$ 
and 
by $(c_*(E_1)c_*(E_2))_i \in H^{2i}(X, \Z)$ the degree $i$ part of the product $c_*(E_1)c_*(E_2)$ of the total Chern classes $c_*(E_i)$ of the topological $\C$-vector bundles $E_i$ ($i=1,$ $2$) on $X.$ Recall that $T_X$ has a natural underlying $\R$-vector bundle structure and we have  canonical isomorphisms:
$$T_X \otimes_{\R} \C \simeq T_X \oplus \overline{T_X} \simeq T_X \oplus T_X^{*}\,\, .$$
Here $\overline{T_X}$ is the $\C$-vector bundle with complex multiplication by the conjugate and $T_X^* \simeq \Omega_X^1$ is the dual holomorphic $\C$-vector bundle of the holomorphic $\C$-vector bundle $T_X.$ 

%Here and hereafter, by $\Z/2$ (resp. by $\Q$), we denote the prime  field of characteristic $2$ (resp. of characteristic $0$). 
We call a non-degreasing sequence of positive integers $ \lambda = (\lambda_i)_{i=1}^{k}$ a partition of a positive integer $n$ if $\sum_{i=1}^{k} \lambda_i = n.$ 

%For instance, the possible partitions of $4$ are 
%$$(4)\,\, ,\,\, (1, 3)\,\, ,\,\, (2, 2)\,\, ,\,\, (1, 1, 2)\,\, ,\,\, (1, 1, 1, 1)\,\, .$$ 

\begin{definition}\label{def11} Let $X$ be a projective manifold of dimension $n.$ 
\begin{enumerate}
\item The $i$-th Stiefel-Whitney class $w_i(X)$ of $X$ is defined by
$$w_i(X) := c_i(X)\,\, {\rm mod}\,\, 2 \in H^{2i}(X, \Z/2)  = H^{2i}(X, \Z) \otimes_{\Z} \Z/2\,\, .$$ 
\item For a partition $ \lambda = (\lambda_i)_{i=1}^{k}$ of $n,$ we define the Stiefel-Whitney number $w_{\lambda}(X)$ with respect to $\lambda$ by  
$$w_{\lambda} := \prod_{i=1}^{k} w_{\lambda_i}(X) \in \Z/2\,\, .$$ 
\item The  $i$-th rational Pontrjagin class $p_i(X)$ of $X$ is defined by 
$$p_i(X) := c_{2i}(T_X \oplus \overline{T_X}) = (c_*(T_X)c_*(T_X^*))_{2i} \in H^{4i}(X, \Q) \,\, .$$ 
%We call the $i$-th rational pontrjagin class just as the $i$-th Pontrjagin class
\item Assume that $n$ is divisible by $2$ in $\Z.$ For a partition $\tau = (\tau_i)_{i=1}^{k}$ of $n/2,$ we define the Pontrjagin number  $p_{\tau}(X)$ with respect to $\tau$ by  
$$p_{\tau}(X) := \prod_{i=1}^{k} p_{\tau_i}(X) \in \Z\,\, .$$ 
The Pontrjagin numbers are actually not only rational numbers but also integers (by definition). We do not define the Pontrjagin number when $n$ is odd. 
\end{enumerate} 
\end{definition} 
According to the literature, $w_i(X)$ (resp $p_i(X)$) is written as $w_{2i}(X)$ (resp. $p_{4i}(X)$ sometimes with different sign convention). 
  
Let $X$ be a projective manifold and we set $c_i := c_i(X)$ and $p_i := p_i(X).$  Then, by definition,
$$ p_1 = -c_1^2 + 2c_2;$$
$$ p_2 = c_2^2 - 2c_1c_3 + 2c_4; $$
$$ p_3 = -c_3^2 - 2 c_1c_5 + 2c_2c_4 + 2c_6$$
for $i =1,$ $2,$ $3.$
%$$ p_4 = 2c_2c_6 - 2 c_1c_7 - 2c_3c_5 + c_4^2 + c_8\,\, ,$$

The following fundamental result is well-known. It is proved (in a more general setting) by Wu (see eg. \cite[Page 131]{MS74}) and Novikov \cite{No65} respectively: 
\begin{theorem}\label{thm11} The Stiefel-Whitney classes and the rational Pontrjagin classes are topological invariants. More precisely, letting $X$ and $Y$ be projective manifolds and $\varphi : X \to Y$ an o-homeomorphism, we have: 
\begin{enumerate}
\item $\varphi^*w_i(Y) = w_i(X)$ in $H^{2i}(X, \Z/2)$ for each integer 
 $i \ge 1.$ 
\item $\varphi^*p_i(Y) = p_i(X)$ in $H^{ 4i}(X, \Q)$ for each integer  $i \ge 1.$ 
\end{enumerate} 
\end{theorem} 
Let $X$ be a projective manifold and  $L$ a holomorphic line bundle on $X.$ We call $L$  {\it $2$-divisible} if there is a holomorphic line bundle $M$ on $X$ such that $L = M^{\otimes 2}$ in the Picard group ${\rm Pic}\, (X).$ Similarly we call an element $x \in H^2(X, \Z)$ {\it $2$-divisible} if there is $y \in H^2(X, \Z)$ such that $x = 2y$ in $H^2(X, \Z).$ 

The following well-known corollary will be used frequently in this note:
\begin{corollary}\label{cor11} Let $X_i$ ($i =0,$ $1$) be simply connected projective manifolds and $\varphi : X_0 \to X_1$ an o-homeomorphism. Then $K_{X_0}$ is  $2$-divisible in ${\rm Pic}\, (X_0)$ if and only if $K_{X_1}$ is  $2$-divisible in ${\rm Pic}\, (X_1).$ 
\end{corollary} 
\begin{proof} The $\Z$-modules $H^2(X_i, \Z)$ ($i=0,$ $1$) are torsion free as $X_i$ are simply connected. Let $\varphi : X_0 \to X_1$ be an o-homeomorphism. Then, by Theorem \ref{thm11}, we have in $H^2(X_0, \Z/2)$:
$$c_1(K_{X_0}) = -c_1(X_0) = -\varphi^{*}c_1(X_1) = 
\varphi^*c_1(K_{X_1})\,\, .$$ 
Since  $\varphi^* : H^2(X_1, \Z/2) \to H^2(X_0, \Z/2)$ is an isomorphism as $\Z/2$-vector spaces, it follows that $c_1(K_{X_0})$ is $2$-divisible in $H^2(X_0, \Z)$ if and only if $c_1(K_{X_1})$ is $2$-divisible in $H^2(X_1, \Z).$ By the assumption of the simply connectedness and the projectivity of $X_i,$ the cycle map 
$$c_1 : {\rm Pic}\, (X_i) \to H^2(X_i, \Z)$$ 
is injective and the cokernel $H^2(X_1, \Z)/{\rm Im}\, c_1$ is torsion free for each $i = 0,$ $1.$ As $K_{X_i} \in {\rm Pic}\, (X_i),$ it follows that $K_{X_i}$ is  $2$-divisible in ${\rm Pic}\, (X_i)$ if and only if $c_1(K_{X_i})$ is  $2$-divisible in $H^{2}(X_i, \Z).$ This implies the result. 
\end{proof}

The next theorem is a special case of the seminal result of Freedman \cite{Fr82} (see also \cite[Page 376]{BHPV04}) in dimension $2$:  
\begin{theorem}\label{thm12} Let $X_i$ ($i =0,$ $1$) be  simply connected projective surfaces. Then $X_i$ ($i=0,$ $1$) are o-homeomorphic if and only if the lattices  $ H^2(X_i, \Z) $ are isomorphic.
\end{theorem} 
\begin{remark}\label{rem14} Let $S$ be a simply connected smooth projective surface. The lattice $H^2(S, \Z)$ is unimodular of signature 
$$\tau(S) := \frac{c_1(S)^2 -  2c_2(S)}{3}\,\, .$$
Moreover, the lattice $H^2(S, \Z)$ is even if and only if $K_S$ is $2$-divisible in $H^2(S, \Z).$ We also note that the isomorphism classes of the lattices $H^2(S, \Z)$ for  simply connected smooth projective 
surfaces $S$ are uniquely determined by the rank as $\Z$-modules, the signature and the parity (even or odd). See eg. \cite[Chapters I, IX]{BHPV04} for more details. 
\end{remark}

We close this section with the following useful propositions.

\begin{proposition} \label{campana} Let $X$ and $Y$ be o-homeomorphic compact complex manifolds of dimension $n.$ Assume that $w_2(X) = 0 $ and therefore  $w_2(Y) = 0, $ too.
Write $K_X = 2L_X $ and $K_Y = 2L_Y. $ Then the following assertions hold.
\begin{enumerate} 
\item $\chi(L_X) = \chi(L_Y);$ 
\item if $n$ is odd or  if  $X$ is Fano, then $\chi(L_X) = \chi(L_Y ) = 0. $
\end{enumerate} 
\end{proposition}

\begin{proof} Assertion (1) follows from Hirzebruch-Riemann-Roch, see \cite[Page 204]{HK57}, or \cite[2.4]{CP94}. If $n $ is odd, (2)  follows from Serre duality. If $X$ is Fano, apply Kodaira vanishing and Serre duality: 
$$ \chi(X,L_X) = (-1)^n h^n(X,L) = (-1)^n h^0(X,L_X) = 0\,\, .$$

\end{proof} 

\begin{proposition} \label{comparechern}  Let $X$ and $Y$ be projective manifolds and $\varphi: Y \to X$ be an o-homeo\-morphim. Assume that $b_2(X) = 1$ and - for simplicity -  that $H^2(X,\mathbb Z)$
has no torsion. 
Let $L_X$ be the ample generator of ${\rm Pic}(X)$ and  $L_Y$ be the ample generator of ${\rm Pic}(Y).$ Then there is an integer $s$ such that
 $$ c_1(Y) = \varphi^*(c_1(X) + 2s c_1(L_X))\,\, .$$
 Furthermore, for all integers $m,$ 
 $$\chi(Y,mL_Y) = \chi(X,(m+s)L_X)\,\, . $$
 \end{proposition} 
 
 \begin{proof} We again refer to \cite[Page 204]{HK57}  and \cite[2.4]{CP94}.
 \end{proof}

\section{Fano Manifolds: Proof  of Theorem \ref{cor:thm1}}\label{s2}

Recall that a  Fano $n$-fold  is said to be  {\it of Fano index} $r = r_X \in \Z_{+}$ if there is a primitive (necessarily) ample class $H \in {\rm Pic}\, (X)$ such that $-K_X = rH$ in ${\rm Pic}\, X.$ 
Then, by \cite{KO73} (see also  \cite[Theorem 1.11, Page 245]{Kol96}), $$1 \le r \le n+1$$ 
and $X = \PP^n$  if $r = n+1$ whereas $X = Q_n \subset \PP^{n+1},$ a smooth quadratic hypersurface if $r=n.$ If $r = n-1,$ then $X$ is said to be a del Pezzo manifold; these varieties
being classified; see \cite{Fj90}. 

Theorem \ref{cor:thm1} will be a consequence of the following  more general

\begin{theorem}\label{thm1} 

Let $X$ be a Fano manifold of dimension $n \geq 3.$ Let $Y$ be a compact K\"ahler manifold o-homeomorphic to $X.$ Suppose that $c_1(Y) = 0$ in  $H^2(Y, \R).$ Then  $Y$ has a splitting  $$Y \simeq Y_1 \times Y_2\,\, ,$$ such that the following holds.
\begin{enumerate}
\item $Y_j$  are simply connected K\"ahler manifolds with trivial canonical bundles, $Y_2$ possibly $0$-dimensional;
\item $Y_1$ is a Calabi-Yau manifold of odd dimension $n_1 \geq 3.$ 
\item If $\dim Y_2 > 0, $ then $n \geq 5.$ 
\end{enumerate} 

%\item  For each integer $n \ge 2,$ there is a Fano 2n-fold $X_n$ which is o-diffeomorphic to a smooth projective $2n$-fold $Y_n$ of general type.

\end{theorem}

\begin{proof} First notice that $X$ being simply connected, so is $Y.$ 
Hence $K_Y$ is trivial and $w_2(Y) = 0.$ Thus $w_2(X) = 0$ and $K_X $ is divisible by $2$ in ${\rm Pic}(X)$ by Corollary \ref{cor11}. 
Write $K_X = 2L_X$ in ${\rm Pic}(X).$
From Proposition \ref{campana}, we defer
$$ 0 = \chi(X,L_X) = \chi(Y,\sO_Y)\,\, .$$ 
By the Beauville-Bogomolov decomposition theorem we may write
$$ Y \simeq \prod Z_j \times \prod S_k\,\, ,$$
where the $Z_j $ are Calabi-Yau and the $S_k$ are hyperk\"ahler manifolds. Since $\chi(Y,\sO_Y)$ is multiplicative and since 
$ \chi(Z,\sO_Z) > 0$ for all hyperk\"ahler manifolds $Z$ and all even-dimensional Calabi-Yau manifolds $Z,$  there must be at least one Calabi-Yau factor $Y_1$ of odd dimension. 
The rest of the product will be $Y_2.$  
Since all  positive dimensional factors of $Y_2$ must have dimension at least $2,$ the remaining statements are clear. 
\end{proof}

\vskip .2cm \noindent 
{\it Proof of Theorem \ref{cor:thm1}.}
 The case $n = 4$ is immediate from Theorem \ref{thm1}(1) and (2). Assume therefore that $n \geq 6.$
By Theorem \ref{thm1} (2), the Beauville-Bogomolov decomposition of $Y$ has at least two positive dimensional factors. Then $b_2(X) = b_2(Y) \ge 2,$ 
a contradiction to our assumption $b_2(X) =1.$ 
%This completes the proof of Theorem \ref{cor:thm1} (1).} 

\section{Fano Manifolds: Proof of Theorem \ref{xxx}}

To prove Theorem \ref{xxx}, we consider o-homeomorphic Fano $n$-folds $X$ and $Y$ with $b_2 = 1$ via an o-homeomorphism 
$$ \varphi: Y \to X$$ and Fano indices $r_X $ and $r_Y.$  Let $L_X $ and $L_Y$ be the ample generators of
the Picard groups so that $-K_X = r_X L_X$ and $-K_Y = r_Y L_Y.$ Of course, it suffices to show $r_X = r_Y,$ since then
$$ c_1(Y)^n = r_Y^n c_1(L_Y)^n = r_X^n c_1(L_X)^n = c_1(X)^n\,\, .$$
To prove equality of the indices, we argue by contradiction and write, using Proposition \ref{comparechern}
$$ c_1(Y) = \varphi^*(c_1(X) + 2sL_X)\,\, , $$
i.e., $r_Y = r_X + 2s.$ We may assume $s < 0.$ 
By Proposition \ref{comparechern}, we also have
$$ \chi(Y,mL_Y) = \chi(X,(m+s)L_X)$$
for all integers $m.$ In particular, setting $m = 0,$ we obtain by Kodaira vanishing and Serre duality
$$1 = \chi(Y,\sO_Y)  = \chi(X,sL_X) = (-1)^n h^n(X,sL_X) = (-1)^n h^0(X,(-r_X -s)L_X)\,\, .$$
However, this is impossible, because $-r_X - s < 0$ by our assumption $s <0.$ 

If finally $X$ is a del Pezzo $n$-fold, so is $Y$ by virtue of $r_Y = r_X.$ Since furthermore $c_1(X)^n  = c_1(Y)^n,$ the del Pezzo $n$-folds $X$ and $Y$ are deformation equivalent.
This is a consequence of the explicit classification of del Pezzo $n$-folds in \cite[Theorem 8.11]{Fj90}.

\section{Cubic Fourfolds: Proof of  Theorem \ref{thm2}}\label{s3}

Let $X \subset \PP^5$ be a smooth cubic $4$-fold. We denote by $H$ (resp. $H_X$) the hyperplane class of $\PP^5$ (resp. the restriction of $H$ to $X$). 
Throughout this section, we denote by $Y$ a smooth projective $4$-fold with an o-homeomorphism $\psi : Y \to X$ and set $L := \psi^*H_X \in H^2(Y, \Z).$ 
We already saw that $c_1(Y) \ne 0 $ (Theorem \ref{cor:thm1})
and that $Y$ is deformation equivalent to $X$ if $Y$ is Fano (Theorem \ref{xxx}).  Hence we assume that $K_Y$ is ample. 
%As $X$ is simply connected, so is $Y.$ 
First we compute several invariants of $X.$

\begin{lemma}\label{lem31} 

\begin{enumerate}
\item $c_1(X) = 3H_X,$ $c_2(X) = 6H_X^2,$ $c_3(X) = 2H_X^3$ and 
$c_4(X) = 27.$ 
\item The Betti numbers $b_k(X)$ are as follows:

$b_0(X) = b_8(X) = 1,$ $b_1(X) = b_3(X) = b_5(X) = b_7(X) = 0,$ 

$b_2(X) = b_6(X) = 1$ and $b_4(X) = 23.$ 

Moreover, ${\rm Pic}\, (X) \simeq H^2(X, \Z) = \Z H_X.$ 
\item The Hodge numbers $h^{p, q}(X)$ of $X$ are as follows:

$h^{0,0}(X) = h^{4,4}(X) = 1,$ $h^{p, q}(X) = 0$ for $p+q = 1, 3, 5, 7,$ 

$h^{2,0}(X) = h^{0,2}(X) = 0,$ $h^{1,1}(X) = 1,$ $h^{2,4}(X) = h^{4,2}(X) = 0,$ $h^{3,3}(X) = 1,$ 

$h^{4,0}(X) = h^{0,4}(X) = 0,$ $h^{3,1}(X) = h^{1,3}(X) = 1$ and $h^{2,2}(X) = 21.$ 

\item The signature  $\tau(X)$ of $(H^4(X, \Z), (*, **)_X)$ is $23.$ 

\end{enumerate}
\end{lemma}

\begin{proof} Consider the standard exact sequence
$$0 \to T_X \to T_{\PP^5}|_X \to N_{X/\PP^5} = \sO_X(3H_X) \to 0\,\, .$$
%Note that the total Chern class is always a unit element in the cohomology ring. Then by the Whitney property and the naturality of the Chern classes, we have an equality of the total Chern classes
This leads to the following equality of Chern classes 
\begin{equation}\label{eq31}
c_*(X) =  c_*(T_{\PP^5})|_X \cdot c_*(\sO_X(3H_X))^{-1}\,\, .
\end{equation}
Notice that $c_*(T_{\PP^5}) = (1 + H)^6$ and 
$$c_*(\sO_X(3H_X))^{-1} = (1+3H_X)^{-1} = 1 - 3H_X + (3H_X)^2 - (3H_X)^3 + (3H_X)^4\,\, .$$
Substituting these two equations into (\ref{eq31}) and comparing the terms of equal degree, we readily obtain assertion (1). 
%For instance, the topological Euler number $c_4(X)$ is computed as 
%$$c_4(X) = 9H_X^4 = 9 (H^4, 3H)_{\PP^5} = 9 \cdot 3 (H^5)_{\PP^5} = 27\,\, .$$
By the Lefschetz hyperplane theorem and Poincar\'e duality, we obtain the values of $b_k(X)$ for $k \ne 4.$ Substituting these into $c_4(X) = 27 $ gives $b_4(X) = 23.$

The middle Hodge numbers $h^{p, q}(X)$ ($p+q = 4$) are computed in \cite[Page 581]{Vo86} based on the Jacobian ring, as claimed in (3). 

Let $r(X)$ be the signature of $(H^4(X, \Z), (*,**)_X).$ By e.g.  \cite[Theorem 6.33]{Vo02},
$$r(X) = \sum_{p, q \ge 0} (-1)^ph^{p, q}(X)\,\, .$$
Putting in the Hodge numbers $h^{p,q}(X),$ we arrive at  $r(X) = 23.$

This completes the proof of Lemma \ref{lem31}.   
\end{proof}

Using Hodge decomposition and Hodge symmetry we immediately deduce
\begin{lemma}\label{lem32} 

\begin{enumerate}

\item The Betti numbers $b_k(Y)$ are as follows:

$b_0(Y) = b_8(Y) = 1,$ $b_1(Y) = b_3(Y) = b_5(Y) = b_7(Y) = 0,$ 

$b_2(Y) = b_6(Y) = 1$ and $b_4(Y) = 23.$ 

Moreover, ${\rm Pic}\, (Y) \simeq H^2(Y, \Z) = \Z L.$
%(Recall that $L = \psi^*H_X$). 

\item The Hodge numbers $h^{p, q}( Y)$ ($p+q \not= 4$) of $Y$ are as follows:

$h^{0,0}(Y) = h^{4,4}(Y) = 1,$ $h^{p, q}(Y) = 0$ for $p+q = 1, 3, 5, 7,$ 

$h^{2,0}(Y) = h^{0,2}(Y) = 0,$ $h^{1,1}(Y) = 1,$ $h^{2,4}(Y) = h^{4,2}(Y) = 0,$ $h^{3,3}(Y) = 1.$

\item The signature  $\tau(Y)$ of $(H^4(Y, \Z), (*, **)_Y)$ is $23.$ 

\end{enumerate}
\end{lemma}

%\begin{proof} As $Y$ is o-homeomorphic to $X,$ the first assertion of (1) and (3) follows from Lemma \ref{lem31}. 
%As $Y$ is smooth projective, the second assertion of (1) and (2) follows from the first assertion of (1) and the Hodge 
%decomposition theorem and the Hodge symmetry applied for $Y.$
%\end{proof}

We introduce the following shorthands
$$h^{4, 0}(Y) = a\,\, ,\,\, h^{3, 1}(Y) = b\,\, ,\,\, h^{2,2}(Y) = c\,\, ,\,\, K_Y = rL \in {\rm Pic}\, (Y)\,\, .$$
%Here $a,$ $b,$ $c$ are non-negative integers and $r$ is an integer and $c_1(Y) = -K_Y = -rL.$ 
Via the isomorphisms $c_1 : {\rm Pic}\, (Y) \simeq H^2(Y, \Z)$ and $c_1 : {\rm Pic}\, (X) \simeq H^2(X, \Z),$ we regard $L \in {\rm Pic}\, (Y)$ and $H_X \in {\rm Pic}\, (X)$ whenever suitable.

\begin{lemma}\label{lem33} 

\begin{enumerate}

\item $b=1$ and $c = 21 - 2a.$ In particular $0 \le a \le 10.$ 

\item $\chi(\sO_Y) = a +1.$ 

\end{enumerate}
\end{lemma}

\begin{proof} 
By our assumption and Hodge symmetry,  $b_4(Y) = b_4(X) $ and 
$a = h^{4, 0}(Y) = h^{0, 4}(Y)$ and 
$b = h^{3, 1}(Y) = h^{1, 3}(Y).$
Hence by Lemma  \ref{lem32} and Hodge decomposition,
\begin{equation} \label{eq32}
2a + 2b + c = b_4(Y) = 23\,\, .
\end{equation}
%As $Y$ is o-homeomorphic to $X,$ the signature 
%$r(Y)$ of $(H^4(Y, \Z), (*,**)_Y)$ is the same as the signature {\blue $r(X)$ o%f $(H^4(X, \Z), (*,**)_X).$} 
By $r(Y) = \sum_{p, q \ge 0} (-1)^ph^{p, q}(Y)$ (\cite{Vo02}) and Lemmas  \ref{lem32}, we obtain
\begin{equation}\label{eq33}
2a -2b + c + 4 = r(Y)  = 23\,\, . 
\end{equation}
Subtracting (\ref{eq33}) from (\ref{eq32}), gives $b=1,$ hence $c = 21 - 2a.$ 
As 
$$\chi(\sO_Y) = \sum_{q=0}^{4} (-1)^{q} h^{0, q}(Y)\,\, ,$$ 
the assertion $\chi(\sO_Y) = 1+a$ follows.
\end{proof}

Next we compute all the Chern numbers in terms of $r.$
\begin{lemma}\label{lem34} 
$$c_1(Y)^4 = 3r^4\,\, ,\,\, c_1(Y)^2c_2(Y) = \frac{3}{2}r^2(r^2+3)\,\, ,$$
$$c_2(Y)^2 = 3(\frac{r^2 + 3}{2})^2\,\, ,\,\, c_1(Y)c_3(Y) = -36 + \frac{3}{2}(\frac{r^2+3}{2})^2\,\, ,$$
and $c_4(Y) = 27.$
\end{lemma}
\begin{proof}
As $Y$ is o-homeomorphic to $X,$ $c_4(Y) = c_4(X) = 27$ by Lemma \ref{lem31}. 
%Consider the Pontr{\red j}agin classes $p_i(X),$ $p_i(Y)$ of 
%$X$ and $Y.$ 
By Lemma \ref{lem31}, we have
\begin{equation}\label{eq34}
p_1(X) = 2c_2(X) - c_1(X)^2  = 3H_X^2\,\, ,
\end{equation}
\begin{equation}\label{eq35}
p_2(X) = 2c_4(X) -2c_1(X)c_3(X) + c_2(X)^2 = 126\,\, .
\end{equation}
%As $\psi^*H_X = L$ and $\psi^*p_1(X) = p_1(Y)$ in $H^*(Y, \Q)$ {\red by Theorem \ref{thm11} (3),}
%by the invariance of the Pntryagin classes under any o-homeomorphism 
%\cite{No65}, 
%it follows from (\ref{eq34}) that
The invariance of the Pontrjagin classes  (Theorem \ref{thm11} (2))
yields
\begin{equation}\label{eq36}
p_1(Y) = 2c_2(Y) -c_1(Y)^2 = 3L^2\,\, ,
\end{equation}
\begin{equation}\label{eq37}
p_2(Y) = 2c_4(Y) - 2c_1(Y)c_3(Y) + c_2(Y)^2 = 126\,\, .
\end{equation}
As $c_1(Y) = -K_Y = -rL,$ it follows from (\ref{eq36}) that 
\begin{equation}\label{eq38}
c_2(Y) = \frac{r^2 + 3}{2}L^2\,\, .
\end{equation}
As $c_1(Y) = -rL$ and $(L^2)^2 = L^4 = H_X^4 = 3,$ the first three equalities in Lemma \ref{lem33} follow from (\ref{eq38}). Then substituting $c_4(Y) = 27$ and $c_2(Y)^2 = 3((r^2+3)/2)^2$ into (\ref{eq37}), we obtain the value $c_1(Y)c_3(Y)$ as claimed.  
\end{proof}

Next we determine $r$ and $a$: 
\begin{lemma}\label{lem35} 
$r=  3$ and $a=0.$ In particular, $K_Y =  3L$ in ${\rm Pic}\, (Y)$ and $h^{p, q}(Y) = h^{p, q}(X)$ for all $p,$ $q.$ 
\end{lemma}
\begin{proof} By the Riemann-Roch formula and Lemma \ref{lem33}  (2), we have
$$a + 1 = \chi(\sO_Y) = -\frac{1}{720}(c_1^4 -4c_1^2c_2 -3c_2^2 -c_1c_3 + c_4)\,\, .$$
By Lemma \ref{lem34}, the right hand side is written in terms of $r$ as follows:
$$a+1 = -\frac{1}{720}(3r^4 - 4\cdot \frac{3(r^2+3)r^2}{2} - 9 (\frac{r^2+3}{2})^2 + 36 - \frac{3}{2}(\frac{r^2+3}{2})^2 + 27)\,\, .$$
We set $R := r^2.$ Then simplifying the right hand side, we finally obtain 
\begin{equation}\label{eq39}
a+1 = \frac{1}{128}((R+3)^2 - 16) = \frac{1}{128}((r^2+3)^2 - 16)\,\, .
\end{equation}
As $a+1$ is a positive integer, it follows from (\ref{eq39}) that $r^2 + 3$ is an even integer and satisfies $r^2 + 3 > 4.$ Thus $r$ is an odd integer such that $r \ge 3.$ 

If $r \ge 7,$ then $r^2 + 3 \ge 52$ and therefore
$$\frac{1}{128}((r^2+3)^2 - 16) \ge \frac{52^2 -{ 16}}{128} = 21 
\ge 12\,\, .$$
On the other hand, as $0 \le a \le 10$ by Lemma \ref{lem33}, we have $a+1 \le 11.$ Thus $r \le 5$ by (\ref{eq39}). As $r \ge 3$ and $|r|$ is an odd integer, it follows that $r= 3$ or $5.$ If $r = 5,$ then $c_1(X) =  5L.$ Then, by Lemma \ref{lem34}, we have
$$ 5 L.c_3(Y) = -36 + \frac{3}{2} \cdot 14^2 = -258\,\, ,$$
a contradiction, as the left hand side is an integer divisible by $5$ { in $\Z$} but the right hand side is not. Hence $r = 3.$ Then $a = 0$ by (\ref{eq39}). Now the last statement follows from Lemma \ref{lem32} and Lemma \ref{lem33}. This completes the proof.  
\end{proof}

This completes the proof of Theorem \ref{thm2}(1). 

In what follows, we show Theorem \ref{thm2} (2). We may and will assume that $\psi$ is a log o-homeomorphism. Recall that $\psi^*H_X = L$ and ${\rm Pic}\, (X) = \Z H_X$ and ${\rm Pic}\, (Y) = \Z L.$ 
As $\psi$ is a log o-homeomorphism, there is then a positive intger $d$ and a smooth $3$-fold $V \in |dH_X|$ and a smooth $3$-fold $W \subset Y$ (necessarily in $|dL|$) such that $\psi|_V : V \to W$ is an o-homeomorphism. In particular, $L$ is the {\it ample} generator of ${\rm Pic}\, (Y).$ 

\begin{lemma}\label{lem36} Assume that $K_Y$ is ample. Then:
$$c_3(V) = 3d(2 -6d + 3d^2 - d^3)\,\, ,\,\, c_3(W) = 3d(-2 -6d - 3d^2 - d^3)\,\, .$$
In particular, $c_3(W) \not= c_3(V).$ 
\end{lemma}
\begin{proof} We set $H_V := H|_{V} = H_X|_V$ and $L_W = L|_W.$
By the standard exact sequence 
$$0 \to T_V \to T_X|_V \to N_{V/X} = \sO_V(dH_V) \to 0$$
and by Lemma \ref{lem32}, we have 
$$c_{*}(V) = (1+3H_V + 6H_V^2 + 2H_V^3)(1 - dH_V + d^2H_V^2 - d^3H_V^3)\,\, .$$
Therefore,
$$c_3(V) = 3d(2 -6d + 3d^2 - d^3)\,\, .$$

By Lemma \ref{lem35}, $c_1(Y) = -K_Y = -3L,$ i.e., $r = -3,$ as $K_Y$ is ample and $L$ is now the ample generator of ${\rm Pic}\, (Y).$ By Lemma  \ref{lem34} and by $r=-3,$ we have that
$$c_2(Y) = \frac{3^2 +3}{2}L^2 = 6L^2\,\, .$$
As $H^{3, 3}(Y, \Z) \simeq \Z$ and 
$$c_3(Y)c_1(Y) = -36 + \frac{3}{2}(\frac{3^2 +3}{2})^2 = 18 = -2L^3.c_1(Y)$$
by $c_1(Y) = -3L$ and $L^4 = 3,$ it follows that 
$$c_3(Y) = -2L^3\,\, .$$
Thus, by the standard exact sequence 
$$0 \to T_W \to T_Y|_W \to N_{W/Y} = \sO_W(dL_W) \to 0$$
we have
$$c_{*}(W) = (1-3L_W + 6L_W^2 - 2L_W^3)(1 - dL_W + d^2L_W^2 - d^3L_W^3)
\,\, .$$
Using $L_W^3 = d(L^3, L)_Y = 3d,$ we obtain
$$c_3(W) = 3d(-2 -6d - 3d^2 - d^3)\,\, .$$
If $c_3(V) = c_3(W),$ then, as $d \not= 0,$ we have then $2 + 3d^2 = 0.$ However, this impossible, as $d$ is an integer. This complete the proof.
\end{proof}
As $W$ is o-homeomorphic to $V,$ it follows from Lemma \ref{lem36} that $Y$ can not be of general type, if $Y$ is log o-homeomorphic to $X.$ This completes the proof of Theorem \ref{thm2} (2).

\section{Del Pezzo Fourfolds of degree  five
} \label{s3a}

We begin with the following general 

\begin{proposition} \label{propindex} 
Let $X$ be a Fano manifold of dimension $n$ with $b_2(X) = 1$ and of Fano index $r.$  Let $L_X$ be the ample generator of $X.$  Assume that $b_{2k}(X) \leq  2$ for $k  > 1$ and that $b_{2k+1}(X) = 0$ for all $k.$ 
 Let $Y$ be a manifold of general type o-homeomorphic to $X.$  Then $n$ is even.
 If furthermore  $$ \dim H^0(X,mL_X) \ne 1$$ for all positive integers $m$ and
 if $L_Y$ denotes the ample generator, then
 $$ K_Y = rL_Y.$$
 \end{proposition} 
 
 \begin{proof} By our assumption on the Betti numbers of $X,$ Hodge decomposition and Hodge symmetry, we obtain
 $$ H^q(Y,\mathcal O_Y) = 0 $$ for $q > 0,$ in particular $\chi(Y,\mathcal O_Y) = 1.$ Let $s$ be as in Propositon \ref{comparechern}, so that 
 $$ \chi(Y,\mathcal O_Y) = \chi(X,sL_X)\,\, .$$ 
Consequently, $\chi(X,sL_X) = 1.$
 Since $s < 0$ by our assumptions, Kodaira vanishing gives 
 $$ (-1)^n h^n(X,sL_X) = 1\,\, .$$ 
 In particular, $n$ is even. By Serre duality, 
  we deduce 
 $$ h^0(X,(-s-r)L_X) = 1\,\, .$$ 
 Hence our assumption on the number of sections in $mL_X$ gives $-s = r,$ hence
 $$-K_Y = \varphi^*(-K_X + 2sL_X) = \varphi^*(-rL_X) = -r \varphi^*(L_X)\,\, .$$
A priori, $\varphi^*(L_X) = \pm L_Y.$ As $r = -s > 0$ and $Y$ is of general type, it follows that $\varphi^*(L_X) = L_Y.$ This proves our claim.
  \end{proof}
  
  As a special case, we have

\begin{corollary} \label{thm2b}  Let $X$ be an  $n$-dimensional Fano manifold of index $r$ with the same Betti numbers as $\mathbb P^n$ or the quadric $Q_n$ and ample generator $L_X.$ 
Assume that $H^0(X,mL_X) \ne 1$ for all positive integers $m.$ 
Let $Y$ be o-homeomorphic to $X$ and assume that $K_Y$ is ample. Then $n$ is even. Moreover, if 
$L_Y$ denotes the ample generator of  $Y,$ then $K_Y = (n+1)L_Y$ resp. $K_Y = nL_Y.$ 
\end{corollary}
  
Our main result in this section is concerned with del Pezzo $4$-folds of degree $5.$

\begin{theorem} \label{thm2a} Let $X$ be a del Pezzo  $4$-fold of degree $5.$ Explicitly, $X$ is a smooth complete intersection of the form
$$X = G(2,5) \cap H_1 \cap H_2\,\, ,$$ 
where $G(2,5) \subset  \mathbb P^9$ is the Grassmannian, embedded by Pl\"ucker, 
and $H_j$ are general hyperplane sections. Assume that the projective manifold $Y$ is o-homeomorphic to $X.$ Then 

\begin{enumerate}
\item $Y$ is either deformation equivalent to $X$ or a smooth projective $4$-fold with ample $K_Y$ and with the same Hodge numbers as $X.$ Moreover all  Chern numbers (see Sections \ref{sprep} and \ref{s6} for the definition)
 on $X$ and $Y$ agree. 

\item Assume in addition that $Y$ is log o-homeomorphic to $X.$ Then $Y$ is isomorphic to a del Pezzo $4$-fold of degree $5.$ 
\end{enumerate}

\end{theorem}

\begin{proof} 
Since our proof is similar to the proof of Theorem \ref{thm2}, we omit standard details. 
Let $$X = G(2,5) \cap H_1 \cap H_2 $$ be a del Pezzo  $4$-fold of degree $5$ and $\varphi: Y \to X$ be an o-homeomorphism from a projective manifold $Y.$
Let $L_X$ and $L_Y$ denote the ample generators. 
 First observe that by \cite[4.7]{PZ16}, 
$$ b_4(Y) =  b_4(X) =  2\,\, .$$ 
Moreover,  $b_1(X) = b_3(X) = 0$ and $b_2(X) = 1$ by Lefschetz.  
Hence by Hodge decomposition and Hodge symmetry, all Hodge numbers of $X$ and $Y$ agree.  
Assume that $Y$ is not Fano. By Theorem \ref{cor:thm1}, $K_Y$ is ample. 
Thus Proposition \ref{propindex} applies and 
$$ K_Y = 3L_Y\,\, .$$ 
Using the invariance of the Pontrjagin  classes and the topological Euler numbers (giving three conditions on the Chern numbers), and the equality
$\chi(Y,\mathcal O_Y)  = \chi(X, \mathcal O_X) = 1$ in connection with  Riemann-Roch (giving two additional conditions on the Chern numbers),
we see that all  Chern numbers, i.e., the quantities 
$$c_1^4\,\, ,\,\, c_1^2c_2\,\, ,\,\, c_1c_3\,\, ,\,\,  c_2^2\,\, ,\,\, c_4$$ 
of  $Y$ and $X$ agree. \\
If $Y$ is Fano, then Theorem \ref{xxx} applies, and $Y$ is deformation equivalent to $X,$ 
This shows Part (1).

Part (2) is demonstrated similarly as in Theorem \ref{thm2}(2).
In fact, if $(X,V)$ and $(Y, W)$ are log o-homeomorphic, then there is a positive integer $d$ such that $V \in |dL_X|$ and $W \in |dL_Y|.$ Suppose $Y$ is of general type. 
In the following, we will identify $H^*(X,\mathbb C) $ and $H^*(Y,\mathbb C)$ and write $L$ for short instead of $L_X$ and $L_Y.$ 
Since $V$ and  $W$ are o-homeomorphic, $c_3(V) = c_3(W).$ 
Using the tangent bundle sequences for $V \subset X$ and $W \subset Y,$ we obtain
$$ c_3(V) = d L \cdot c_3(X) - d L_V c_2(V)$$
and 
$$ c_3(W) = d L \cdot c_3(Y) - d L_W c_2(W)\,\, .$$
By the equality of Chern numbers, we have $L \cdot c_3(X) = L \cdot c_3(Y), $ hence
$$ L_V \cdot c_2(V) = L_W \cdot c_2(W)\,\, .$$ 
Using again the tangent bundle sequences, we derive
$$ L_V c_2(V) = d L^2 c_2(X) + (d-3)d^2 L^4 $$
and 
$$ L_W c_2(W) =  dL^2 c_2(Y) + (d+3) d^2 L^4\,\, .$$ 
Again by the equality of Chern numbers, we have $L^2 c_2(X) = L^2 c_2(Y), $ hence
$$ (d-3)d^2 = (d+3)d^2\,\, ,$$
which is absurd. This completes the proof. 
\end{proof}

\section{Ricci-flat Manifolds: Proof of Theorem \ref{thm3}}\label{s4}

\noindent 
{\it Proof of Theorem \ref{thm3}(1).} \\
Let $X$ be a Calabi-Yau $2n$-fold and $Y$ a hyperk\"ahler $2n$-fold with  $n \geq 2.$ 
Assume that there is an o-homeomorphism $$\varphi : X \to Y\,\, .$$
By definition, $K_X = 0$ and $K_Y = 0$ in the Picard groups ${\rm Pic}\, (X)$ and ${\rm Pic}\, (Y).$ Hence $w_2(X) = 0$ and $w_2(Y) = 0$ and by Proposition  \ref{campana} (1), we obtain 
$$ \chi(X,\mathcal O_X) = \chi(Y,\mathcal O_Y)\,\, .$$  
However, this
%Using the invariance of the Pontrjagin classes and observing that $\chi(Z,\mathcal O_Z)$ for a compact manifold with $c_1(Z) = 0 $
%depends only on the Pontrjagin classes, see e.g.  \cite[Page 204]{HK57} and Proposition \ref{comparechern}, we obtain 
%$$ \chi(X,\mathcal O_X) = \chi(Y,\mathcal O_Y).$$  
is in contradiction to 
$$\chi(\sO_Y) = \sum_{p} (-1)^ph^{0, p}(Y) = \sum_{p} (-1)^ph^{p, 0}(Y)  = n+1 \geq 3$$ 
for any hyperk\"ahler $2n$-fold $Y$  with $n \ge 2,$ while 
$$\chi(\sO_X) = \sum_{p} (-1)^ph^{0, p}(Y) = \sum_{p} (-1)^ph^{p, 0}(Y) = 2$$ 
for any Calabi-Yau $2n$-fold $X.$
%Therefore Calabi-Yau $4$-folds and hyperk\"ahler $4$-folds are never o-homeomorphic. 

\vskip .2cm \noindent {\it  Proof of Theorem \ref{thm3}(2).}  \\
By our assumption, $\pi_1(X) = 0.$ 
Let 
$$ X \simeq  \prod_{j=1}^r  M_j \times \prod_{k=1}^s N_k $$
be the Beauville-Bogomolov decomposition  (\cite{Be83}) with $M_j$ Calabi-Yau  manifolds and $N_k$ hyperk\"ahler. 
Since $\chi(Y,\mathcal O_Y) = n+1$ and since $\chi(X,\mathcal O_X) = \chi(Y,\mathcal O_Y)$ by Proposition \ref{campana} (1),
we obtain 
$$ \chi(X,\mathcal O_X)  =  n+1\,\, .$$
In particular, $\chi(M_j, \mathcal O_{M_j}) \not= 0.$ Hence the Calabi-Yau manifolds $M_j$ are of even dimension $\ge 2$ and we have $\chi(M_j, \mathcal O_{M_j}) = 2.$ Therefore
$$ n+1 = \chi(X,\mathcal O_X) =  \prod_{j=1}^r \chi(M_j, \mathcal  O_{M_j}) \times \prod_{k=1}^s \chi(N_k,\mathcal O_{N_k}) =  2^r \cdot  \prod_{k=1}^s (n_k+1)\,\, ,$$ 
where $ 2n_k = \dim N_k.$ Since $n$ is even, we must have $r = 0.$ Then 
$s = 1.$ 
Indeed, by subtracting 
$$n = \sum_{k=1}^{s}n_k$$
from 
$$ n+1 = 2^0 \cdot \prod_{k=1}^{s}(n_k+1) = \prod_{k=1}^{s}(n_k+1)\,\, ,
$$
it follows that $0 \ge n_1n_2 \ldots n_s$ if $s \ge 2,$ a contradiction to $n_k \ge 1.$ 
This completes the proof of Theorem \ref{thm3}  (2).

\section{Hyperk\"ahler Fourfolds: Proof of Theorem \ref{thm4}}\label{s5}
Using the same notation as in Theorem \ref{thm4}, 
we may assume $X = S^{[2]}$ for a projective K3 surface $S.$ Let $Y$ be a smooth projective $4$-fold being o-homeomorphic to $X$ via an o-homeomorphism $\varphi : X \to Y.$ 

The next lemma is well-known (see eg. \cite{Gu01}):

\begin{lemma}\label{lem51} 

\begin{enumerate}

\item The Betti numbers $b_k(X)$ and the topological Euler number $c_4(X)$ 
of $X$ are as follows:

$b_0(X) = b_8(X) = 1,$ $b_1(X) = b_3(X) = b_5(X) = b_7(X) = 0,$ 

$b_2(X) = b_6(X) = 23$ and $b_4(X) = 276,$ in particular, $c_4(X) = 324.$

\item The Hodge numbers $h^{p, q}(X)$ of $X$ are as follows:

$h^{0,0}(X) = h^{4,4}(X) = 1,$ $h^{p, q}(X) = 0$ for $p+q = 1, 3, 5, 7,$ 

$h^{2,0}(X) = h^{0,2}(X) = 1,$ $h^{1,1}(X) = 21,$ $h^{2,4}(X) = h^{4,2}(X) = 1,$ $h^{6,6}(X) = 21,$ 

$h^{4,0}(X) = h^{0,4}(X) = 1,$ $h^{3,1}(X) = h^{1,3}(X) = 21$ and $h^{2,2}(X) = 232.$ 

\item The cup product,  i.e., the  multiplication map
 $${\rm Sym}^2H^2(X, \Q) \to H^4(X, \Q)$$ is an isomorphism.

\item There is an integral quadratic form $q_X(x) \in {\rm Sym}^2(H^2(X, \Z)^*)$ such that 
$$E^4 = 3 q_X(E)^2$$ 
for all $E \in H^2(X, \Z).$ Moreover, if $q_X(E) = 0,$ then $ E^3 = 0$ in $H^6(X, \Z).$ 

\end{enumerate}
\end{lemma}

\begin{proof} For the convenience of the reader, we sketch an outline of proof. The first assertion  of (4) is the so called Fujiki's relation, stating the following. Given  a hyperk\"ahler $2n$-fold $M,$  
 let $q_M(x)$ be the Beauville-Bogomolov-Fujiki form. Then there is a positive constant $c_M$ such that 
$$(x^{2n})_M = c_Mq_M(x)^{n}$$ for all $x \in H^2(M, \Z).$ See \cite{Fu87} and  \cite[Proposition 23.14]{GHJ03}  for details.

Let ${\rm Bl}_{\Delta_S} (S \times S)$ be the blow up of $S \times S$ along the diagonal $\Delta \simeq S$ and $\tilde{\iota}$ the involution of ${\rm Bl}_{\Delta_S} (S \times S)$ induced by the involution $(P, Q) \mapsto (Q, P)$ on $S \times S.$ Then 
$$X = S^{[2]} = {\rm Bl}_{\Delta_S} (S \times S)/ \tilde{\iota}\,\, .$$ 
The assertion (1) follows from a standard calculation based on the K\"unneth formula applied {\blue to} $S \times S.$ By $b_k(S) = 0,$ and hence $b_k(S \times S) = 0$ if $k$ is odd, it follows that $b_k(X) = 0$ if $k$ is odd. We have  
$$\chi_{\rm top}(X) = \frac{\chi_{\rm top}(S) \times \chi_{\rm top} (S) - \chi_{\rm top} (S)}{2} + 2\chi_{\rm top} (S) = \frac{24 \times 23}{2} + 48 = 324\,\, .$$
Combining this with $b_{6}(S^{[2]}) = b_2(S^{[2]}) = b_2(S) + 1 = 23,$ we obtain $b_4(S^{[2]}) =  276.$ 

Let $M$ be a smooth projective manifold of dimension $2n.$ Consider the natural graded ring homomorphism 
$$m := \oplus m_{{ k}} : \oplus_{{ k} \ge 0} {\rm Sym}^{{ k}}H^2(M, \Q) \to \oplus_{n \ge 0}  H^{2k}(M, \Q)\,\, ,$$
 given by the cup product. { Given a} hyperk\"ahler manifold $M$ { of dimension $2n$}, Verbitsky (\cite{Ve96}, see also \cite[Proposition 24.1]{GHJ03}), shows that the  graded ring homomorphism { $m$ is} injective in { degree $2k \le 2n$} and the kernel, 
which is in { degree $2k > 2n$}, is the graded ideal generated by all the elements $x^{{ n}+1}$ with $x \in H^2(M, \Q)$ such that $q_{ M}(x) = 0.$ 

Apply this for our $X.$ { Let $x \in H^2(X, \Z).$} If $q_X(x) = 0,$ then $x^3 = 0,$ as $\dim\, X = 4.$ This shows the last assertion of (4). We also have $b_4(X) = 276$ and $b_2(X) = 23.$ So, 
$$\dim\, {\rm Sym}^2H^2(X, \Q) = (23+1)\cdot 23/2 = 276 = b_4(X)\,\, .$$ 
Hence ${\rm Sym}^2H^2(X, \Q) = H^4(X, \Q).$ 

Note that $h^{2, 0}(X) = h^{0, 2}(X) = 1,$ $h^{1, 1}(X) = h^{1, 1}(S) + 1 = 21$ and $b_{p+q}(X) = 0,$ whence $h^{p, q}(X) = 0,$ when $p+q$ is odd. Hence the assertions (2) follows from (3) by the same argument as the proof of Lemma \ref{lem52} below. 
\end{proof}

The Betti numbers $b_k(Y)$ and the topological Euler number $c_4(Y)$ are the same as $b_k(X)$ and $c_4(X),$ as $Y$ is o-homeomorphic to $X.$ However, in general, Hodge numbers are not preserved even under o-diffeomorphisms. Nevertheless we have 

\begin{lemma}\label{lem52} 

\begin{enumerate}

\item The Hodge numbers $h^{p, q}(Y)$ of $Y$ are the same as the Hodge numbers of $X,$ i.e., 

$h^{0,0}(Y) = h^{4,4}(Y) = 1,$ $h^{p, q}(Y) = 0$ for $p+q = 1, 3, 5, 7,$ 

$h^{2,0}(Y) = h^{0,2}(Y) = {\blue 1},$ $h^{1,1}(Y) = 21,$ $h^{2,4}(Y) = h^{4,2}(Y) = {\blue 1} ,$ $h^{6,6}(Y) = 21,$  

$h^{4,0}(Y) = h^{0,4}(Y) = 1,$ $h^{3,1}(Y) = h^{1,3}(Y) = 21$ and $h^{2,2}(Y) = 232.$ 

{ \item  The multiplication map 
$${\rm Sym}^2H^2(Y, \Q) \to H^4(Y, \Q)$$ is an isomorphism.} 

\item There is an integral quadratic form  $q_Y(x)$ on $H^2(Y, \Z)$ such that $(E^4)_Y = 3 q_Y(E)^2$ for all $E \in H^2(Y, \Z)$ { and $E^3 = 0$ if $q_Y(E) = 0.$} In particular, $(E^4)_Y$ is 
a non-negative integer for any $E \in H^2(Y, \Z).$ 

\item $K_Y$ is $2$-divisible in ${\rm Pic}\, (Y).$ 

\end{enumerate}
\end{lemma}

\begin{proof}

As $Y$ is homeomorphic to $X,$ it follows that $b_k(Y) = b_k(X)$ for each non-negative integer $k.$ Thus $b_1(Y) = 0.$ As $Y$ is smooth projective manifold, it follows from the 
Hodge decomposition theorem that $h^{1, 0}(Y) = h^{0,1}(Y) = 0.$ Similarly, $h^{p, q}(Y) = 0$ for $p+q = 3, 5, 7$ and $h^{0, 0}(Y) = h^{4, 4}(Y) = 1.$ 

Set $a = h^{2, 0}(Y) = h^{0, 2}(Y)$ and $b = h^{1, 1}(Y).$ By Hodge symmetry, $h^{2, 0}(Y) = h^{0, 2}(Y).$ As $b_2(Y) = b_2(X) = 23,$ we
obtain 
$$2a + b = 23,$$
again by  Hodge decomposition. By  duality, $a = h^{4, 2}(Y) = h^{2, 4}(Y)$ and $h^{3, 3}(Y) = b.$
As $Y$ is homeomorphic to $X,$ the multiplication map $${\rm Sym}^2H^2(Y, { \Q}) \to H^4(Y, {\Q})$$ 
is an isomorphism by Lemma \ref{lem51} (3). Thus  
$$H^4(Y, \Z) = {\rm Sym}^2H^2(Y, \C) = {\rm Sym}^2 (H^{2, 0}(Y) \oplus H^{1, 1}(Y) \oplus H^{0, 2}(Y))\,\, .$$
It follows that 
$$ H^{4, 0}(Y) = {\rm Sym}^2 H^{2, 0}(Y)\,\, ,\,\, H^{0, 4}(Y) = {\rm Sym}^2 H^{0, 2}(Y)\,\, ,$$
$$H^{3, 1}(Y) = H^{2, 0}(Y) \otimes H^{1, 1}(Y)\,\, ,\,\, H^{1, 3}(Y) = H^{0, 2}(Y) \otimes H^{1, 1}(Y)\,\, ,$$
$$H^{2, 2}(Y) = (H^{2, 0}(Y) \otimes H^{0, 2}(Y)) \oplus {\rm Sym}^2 H^{1, 1}(Y)\,\, .$$
Hence 
$$h^{4, 0}(Y) = h^{0, 4}(Y) = \frac{a(a+1)}{2}\,\, ,\,\, h^{3, 1}(Y) = h^{1, 3} = ab = a(23-2a)\,\, ,$$
 $$h^{2, 2}(Y) = \frac{b(b+1)}{2} + a^2 = (23-2a)(12-a) + a^2\,\, .$$
%Here we also used $2a + b = 23,$ i.e., $b = 23 -2a.$ 
So, all $h^{p, q}(Y)$ are expressed in terms of $a.$ We determine $a.$ 

As $Y$ is o-homeomorphic to $X,$ the signature $r(Y)$ of $(H^4(Y, \Z), (*, **)_Y)$ is the same as the signature $r(X)$ of $(H^4(X, \Z), (*, **)_X).$ 

By \cite{Vo02}, we know 
$$r(Y) = \sum_{p, q \ge 0} (-1)^ph^{p, q}(Y)\,\, ,\,\, r(X) = \sum_{p, q \ge 0} (-1)^ph^{p, q}(X)\,\, .$$
Substituting $h^{p, q}(Y)$ above and $h^{p, q}(X)$ in Lemma \ref{lem51} (2), we obtain 
$$r(Y) = 8a^2 - 84a + 232\,\, ,\,\, r(X) = 156\,\, .$$
Thus $8a^2 - 84a + 232 = 156$ by $r(Y) = r(X).$ Regarding this as a quadratic equation of $a$ and solving this equation, we readily find that $a = 1$ or $a = 19/2.$ As $a$ is a non-negative integer, it follows that $a = 1.$ Substituting $a=1$ into $h^{p, q}(Y)$ above, we obtain the assertion (1).

The second assertion follows from the fact that $\varphi : X \to Y$ is an o-homeomorphism. Indeed, we have by Lemma \ref{lem51} (4) that
$$(y^4)_Y = (\varphi^*(y)^4)_X = 3q_X(\varphi^*(y))^2\,\, ,$$
for all $y \in H^2(Y, \Z).$ So, we may define $q_Y(y) := q_X(\varphi^*(y)).$ This proves (2). 

%Recall that the second Stiefel-Whitney classes 
%$w_2(X) = c_1(X)\, {\rm mod}\, 2$ in $H^2(X, \Z/2)$ maps to $w_2(Y) = c_1(Y)\, {\rm mod}\, 2$ in $H^2(Y, \Z/2)$ under any homeomorphisms. As $c_1(X) = 0,$ 
%it follows that $-K_Y = c_1(Y)$ is $2$-divisible in $H^2(Y, \Z).$ 
%The cycle map $c_1 : {\rm Pic}\, (Y) \to H^2(Y, \Z)$ is injective by 
%$h^{0, 1}(Y) = 0.$ Moreover, the image of $c_1$ is primitive in $H^2(Y, \Z)$ 
%by the Lefschetz $(1, 1)$-theorem. Hence the canonical class 
%$K_Y \in {\rm Pic}\, (Y)$ is $2$-divisible in ${\rm Pic}\, (Y)$ as well. 

As $K_X$ is divisible by $2,$ the assertion (3) follows from Corollary \ref{cor11}.
\end{proof}

By Lemma \ref{lem52}, $h^0(Y, \sO(K_Y)) = h^{4, 0}(Y) = 1.$ Hence $\kappa (Y) \not= -\infty.$ This proves Theorem \ref{thm4} (1). 

Next we show that $K_Y$ is nef. We summarize our knowledge on $Y$ as follows.
\begin{enumerate} 
\item $\kappa(Y) \ge 0$ by Theorem \ref{thm4} { (1)} 
\item $K_Y$ is $2$-divisible in ${\rm Pic}\, (Y)$ by Lemma \ref{lem52} { (4)}
\item  $(E^4)_Y \ge 0$ for all $E \in H^2(Y, \Z)$ by Lemma \ref{lem52} { (3)}
\item $b_3(Y) = b_3(X) = 0$ by Lemma \ref{lem51} (1). 
\end{enumerate} 

Therefore  the nefness of $K_Y$ follows from the next slightly more general

\begin{proposition}\label{prop53} Let $Z$ be a smooth projective $4$-fold such that

\begin{enumerate}

\item $\kappa (Z) \ge 0$;

\item $K_Z$ is $2$-divisible in ${\rm Pic}\, (Z)$;

\item $(E^4)_Z \ge 0$ for any $E \in {\rm Pic}\, (Z)$; and

\item $b_3(Z) = 0.$

\end{enumerate}
Then $K_Z$ is nef. 
\end{proposition} 

\begin{proof} Assume to the contrary that $K_Z$ is not nef. Then there is an extremal rational curve $C$ such that $K_Y \cdot C < 0,$ and we have a contraction map $\tau : Z \to Z_1$ of the extremal ray $\R_{\ge 0} C.$ As $\kappa (Z) \ge 0,$ the contraction $\tau$ is either a small contraction or a divisorial contraction. Let $E$ be the exceptional set of $\tau$ (see eg. \cite{Ka84}, \cite{KMM87}, \cite{KM98} for basic result about extremal contractions).

If $\dim E \le 2,$ i.e., $\tau$ is small, then $E$ is a disjoint union of $E_i \simeq \PP^2$ with normal bundle $\sO_{E_i}(-1)^{\oplus 2}$ by a result of Kawamata \cite{Ka89}. Let $H_i$ be the hyperplane class of $E_i = \PP^2.$ Then by the adjunction formula, we have
%$$-3H_i = K_{E_i} = K_Z|_{E_i} - 2H_i\,\, ,{\rm i.e.}\,\, , -H_i = K_Z|_{E_i}\,\, .$$
$$ -H_i = K_Z|_{E_i}\,\, .$$
Now the left hand side is not divisible by $2$ while the right hand side is divisible by $2$ in ${\rm Pic}\, (E_i) = \Z H_i,$ a contradiction. 
Hence $\tau$ has to be a divisorial contraction and $E$ is an irreducible divisor on $Z$ such that $0 \leq \dim \tau(E) \leq 2.$ 
\begin{claim}\label{cl51} $\dim \tau (E) \ne 0.$
\end{claim}
%\vskip .2cm \noindent
%{\blue {\it Claim. $\dim \tau (E) \ne 0.$}} \\ 
%We first show that $\dim\, \tau(E) = 1.$
%{\blue Proof.
\begin{proof} Assume to the contrary that $\dim\, \tau(E) = 0.$ 
Since $\rho(X) = \rho(Z)  +1,$ the normal bundle $N_{E/X} { = \sO_E(E)}$ is negative, i.e., its dual is ample; see \cite[p.354]{An85}. Therefore
{ $$(E^4)_Z = ((E|_E)^3)_E < 0\,\, ,$$}
%$$(E^4)_Z = -((-E|_E)^3)_E < 0\,\, ,$$ 
contradicting assumption (3).
\end{proof}

%Then by a classification due to Ando \cite{An85}, $E$ is isomorphic to either $\PP^3,$ $Q^3 \subset \PP^4,$ a quadratic 3-fold not necessarily smooth, or a Gorenstein Del Pezzo $3$-folds $V.$ More %precisely, Ando \cite[Page 354]{An85} shows that there is a primitive ample line bundle $L \in {\rm Pic}\, (Z)$ such that $\sO_E(-K_Z) = \sO_E(pL)$ and $\sO_E(-E) = \sO_E(qL)$ for some positive %integers $p$ and $q.$ He then used the theory of $\Delta$-genus due to Fujita (see eg. \cite{Fj90}), to conclude that $(E, (p, q))$ is either one of:
%$$(\PP^3, (3, 1))\,\, ,\,\, (\PP^3, (2, 2))\,\, ,\,\, (\PP^3, (1, 3))\,\, ,\,\,$$ 
%$$(Q^3, (2, 1))\,\, ,\,\, (Q^3, (1, 2))\,\, ,\,\, (V, (1,1))\,\, .$$
%As $K_Z$ is $2$-divisible, it follows that we have either $E = \PP^3$ and $-E|_E = 2H$ or $E = Q^3$ and $-E|_{E} = H.$ Here $H$ is the hyerplane class of $\PP^3$ (resp. the hyperplane class of %%^3 \subset \PP^4$). However, then 
%$$(E^4)_Z = -((-E|_E)^3)_E < 0\,\, ,$$ 
%a contradiction to our assumption that $(E^4)_Z \ge 0.$ Hence $\dim\, \tau(E) \not= 0.$ 
\begin{claim}\label{cl52} $\dim \tau (E) \ne 2.$
\end{claim}
%{\blue \vskip .2cm \noindent 
%{\it Claim. $\dim \tau(E) \ne 2.$} \\
%Proof. } 
\begin{proof} Assume that $\dim\, \tau(E) = 2$ and let $F$ be a general fiber of $\tau|_E : E \to \tau(E).$ Then $\dim\, F =1$ and $\tau|_E : E \to \tau(E)$ is equidimensional over some Zariski open subset $ U \subset \tau(E)$ containing $\tau(F).$ 
We set $E_U = (\tau|_E)^{-1}(U).$ By \cite[Proof of Theorem 2.3]{An85} (see also \cite{AW98}), $F  \simeq \PP^1$ with normal bundle $$N_{F/Z} \simeq \sO_F^{\oplus 2} \oplus \sO_F(-1).$$ Hence by adjunction, 
$$-2 =\, {\rm deg}\, K_{F} =\, {\rm deg}\, K_Z|_F - 1\,\, .$$
Thus ${\rm deg}\, K_Y|_F = -1,$ a contradiction to the $2$-divisibility of $K_Z.$ {Hence} $\dim \tau(E) \not= 2.$ 
\end{proof}
%\vskip .2cm 
{\it In conclusion,} $\dim\, \tau(E) = 1.$ 
\begin{claim}\label{cl53} { $Z_1$ is smooth and} $\tau : Z \to Z_1$ is the blow-up along a smooth rational curve $C_1 = \tau(E).$
\end{claim}
%{\blue \vskip .2cm \noindent 
%{\it Claim. $\tau : Z \to Z_1$ is the blow-up along a smooth rational curve $C_%1 = \tau(E).$} \\
%Proof. 
\begin{proof}
Let $F$ be a general fiber of $\tau|_{E} : E \to \tau(E).$ 
%By a result of H. Takagi \cite[Main Theorem, Theorem 1.1]{Ta99} and the $2$-divisibility of $K_Z,$ we see that $\tau : Z \to Z_1$ is the blow-up along a smooth curve $C = \tau(E).$ 
By the $2$-divisibility of $K_Z$ and by a result of Takagi, \cite[Theorem 1.1]{Ta99}, we see that $F = \PP^2$ and $-K_Z|_{F} = 2H$ where $H$ is the hyperplane class of $\PP^2.$ Then $-E|_F = H$ by adjunction. Thus, $\tau$ is the case (4p1) in \cite[Main Theorem]{Ta99}, i.e., $Z_1$ and the curve 
$\tau(E)$ are smooth. Furthermore, $\tau : Z \to Z_1$ is the { blow-up} along the curve 
$\tau(E) \subset Z_1$ and $E$ is the exceptional divisor of the blow-up. 
Set $E = E_1$ and $C_1 = \tau(E_1).$ 
%We next show that $C_1 = \PP^1.$ 
%{\red By \cite[Theorem 7.31]{Vo02}}, we find
%$$H^*(Z, \Z) \simeq H^*(Z_1, \Z) \oplus H^*(E_1, \Z)/H^*(C_1, \Z)$$ 
%as $\Z$-module (see eg. \cite{GH78}). In particular, 
Since 
$$b_3(Z) = b_3(Z_1) + b_3(E_1) = b_3(Z_1) + 2g(C_1)\,\, , $$
see e.g. \cite[Theorem 7.31]{Vo02}, 
and since $b_3(Z) = 0$ by assumption, it follows that $g(C_1) = 0,$ whence $C_1 = \PP^1,$ and $b_3(Z_1) = 0.$ 
\end{proof}

%\vskip .2cm \noindent 

Note that the smooth projective $4$-fold $Z_1$ also satisfies the same properties (1), (2), (3), (4) in Proposition \ref{prop53} as $Z.$ The property (1) for $Z_1$ is clear by the birational invariance of the Kodaira dimension, (2) for $Z_1$ follows from  $K_Z = \tau^* K_{Z_1} + 2E_1$ with the fact that $\tau^*$ is injective, (3) for $Z_1$ follows from $(x^4)_{Z_1} = (\tau^*(z)^4)_Z$ for all $z \in {\rm Pic}\, (Z_1)$ and (4) is already shown above. 
%{\blue 
%\vskip .2cm \noindent 
%{\it Claim. $K_Z$ is nef.}\\
%Proof. } 
\begin{claim}\label{cl54} $K_Z$ is nef.
\end{claim}
\begin{proof} { Assume to the contrary that $K_Z$ is not nef. Then we have an extremal contraction $\tau_0 : Z \to Z_1$ in Claim \ref{cl53}. If} $K_{Z_1}$ is not { yet} nef, then any extremal contraction $\tau_1 : Z_1 \to Z_2$ is again a blow-up along a smooth rational curve of a smooth projective $4$-fold $Z_2,$ and the properties (1), (2), (3), (4) in Proposition \ref{prop53} 
hold for $Z_2$ for the same reason as above. As $\rho(Z) < \infty,$ after finitely many repetation of this process, we finally reach the situation 
$$\tau_n : Z_n \to Z_{n+1}$$
such that $Z_n$ and $Z_{n+1}$ are smooth projective $4$-folds and $\tau_n$ is the blow-up along $C_{n} = \PP^1.$ The properties (1), (2), (3), (4) in Proposition \ref{prop53} still hold for $Z_n$ and $Z_{n+1}$ 
and additionally $K_{Z_{n+1}}$ is nef. Let $E_n$ be the exceptional divisor of $\tau_n.$ Then $\tau(E_n) = C_n.$ As $C_n \simeq \PP^1$ and $\tau$ is the blow-up of $Z_{n+1}$ along $C_n,$ it follows that $$E_n \simeq \PP(N_{C_{n}/Z_{n+1}}^{*})$$ and $E_{n}|_{E_{n}} = \sO_{E_n}(-1).$ Then 
$$(E_n^4)_{Z_n} = ((E_{n}|_{E_{n}})^3)_{E_n} =  -c_1(\sO_{E_{n}}(1))^3\,\, .$$
 By the very definition of the Chern classes, we have 
$$c_1(\sO_{E_{n}}(1))^3 - \pi^*c_1(N_{C_{n}/Z_{n+1}}^*)c_1(\sO_{E_{n}}(1))^2 = 0\,\, .$$
Here $\pi : \PP(N_{C_{n}/Z_{n+1}}^{*}) \to C_n$ is the natural projection. 
Thus,  
$$c_1(\sO_{E_{n}}(1))^3 = {\rm deg}\, N_{C_{n}/Z_{n+1}}^* = -{\rm deg}\, N_{C_{n}/Z_{n+1}}\,\, .$$
Hence, by (3) applied for $Z_n,$ we have:
$$0 \le (E_n^4)_{Z_n} = {\rm deg}\, N_{C_{n}/Z_{n+1}}\,\, .$$
However, as $C_n = \PP^1$ and $K_{Z_{n+1}}$ is nef, we have by the adjunction formula that
$$-2 = {\rm deg}\, K_{C_{n}} = (K_{Z_{n+1}}.C_{n+1})_{Z_{n+1}} + {\rm deg}\, N_{C_{n}/Z_{n+1}} \ge 0\,\, ,$$
a contradiction. Therefore $K_Z$ has to be nef.
\end{proof}

%As $Y$ is homeomorphic to $X,$ it follows that $Y$ is simply connected. 
If $K_Y$ is numerically trivial, then { $Y$ is a hyperk\"ahler $4$-fold by Theorem \ref{thm3} (2), as so is $S^{[2]}.$ }
%having in mind that $Y$ is simply connected,  
%by the Bogomolov-Beauville decomposition theorem \cite{Be83},  
%$Y$ is isomorphic to either a Calabi-Yau $4$-fold, the product of 
%two K3 surfaces or a compact hyperk\"ahler $4$-fold. 
%As $h^{2, 0}(Y) \not= 0$ and $h^{4, 0}(Y) = 1$ by Lemma \ref{lem52}, the first %two cases are not possible. Hence $Y$ is isomorphic to a compact hyperk\"ahler %$4$-fold if $K_Y$ is numerically trivial. 

\vskip .2cm \noindent 
In order to show that $Y$ is not of general type, we first prove  

\begin{proposition}\label{prop54} 
$$(4c_2(Y)-c_1(Y)^2)c_1(Y)^2 = 0\,\, .$$
\end{proposition} 

\begin{proof} 
By Lemma \ref{lem51}, \ref{lem52}, and the Serre duality, $c_4(Y) = c_4(X) = 324$ and $\chi(\sO_Y) = \chi(\sO_X) = 3.$ 
%Here $c_4(X)$ and $c_4(Y)$ are {\blue the} topological Euler numbers of $X$ 
%and $Y,$ as both are smooth compact complex $4$-folds. 
%Note that $T_X \simeq \Omega_X^1$ as $X$ has an everywhere non-degenerate 
%holomorphic $2$-form $\eta_X.$ Thus $c_1(X) = 0$ and $c_3(X) = 0.$ 
%Hence 
By the Riemann-Roch formula for $X,$ we have
$$3 = \chi(\sO_X) = -\frac{1}{720}(-3c_2(X)^2 + c_4(X))$$
and therefore $c_2(X)^2 = 828.$
By the Riemann-Roch formula applied to $Y,$ we have
\begin{equation}\label{eq1}
3 = \chi(\sO_Y) = -\frac{1}{720}(c_1(Y)^4 - 4c_1(Y)^2c_2(Y) -3c_2(Y)^2 - c_1(Y)c_3(Y) + 324)\,\, .
\end{equation} 
Let $\psi : Y \to X$ be an o-homeomorphism. 
%Consider} the Pontrjagin classes  $p_i(X).$ Then, by Theorem \ref{thm11} (3), we have
%$$p_1(Y) = \psi^*p_1(X)\,\, ,\,\, p_2(Y) = \psi^*p_2(Y)\,\, .$$
%By $c_1(X) = 0$ and $c_4(X) = c_4(Y),$ we obtain
The invariance of Pontrjagin classes (Theorem \ref{thm11} (3)) gives
\begin{equation}\label{eq2}
c_1(Y)^2 -2c_2(Y) = -2\psi^*c_2(X)\,\, ,
\end{equation} and
\begin{equation}\label{eq3}
-2c_1(Y)c_3(Y) + c_2(Y)^2 = \psi^*(c_2(X))^2 = 828\,\, , \,\,  {\rm i.e.}\,\, , \,\, c_1(Y)c_3(Y) = \frac{c_2(Y)^2}{2} - 414\,\, .
\end{equation}
 Taking self-intersection on both sides of (\ref{eq3}), we obtain 
$$((c_1(Y)^2 -2c_2(Y))^2)_Y = 4\psi^*(c_2(X)^2)_X = 4 \cdot 828\,\, .$$
Consequently
\begin{equation}\label{eq4}
c_2(Y)^2 = 828 + c_1(Y)^2c_2(Y) - \frac{1}{4}c_1(Y)^4\,\, .
\end{equation}
Substituting (\ref{eq4}) into (\ref{eq3}), we obtain 
\begin{equation}\label{eq5}
c_1(Y)c_3(Y) = \frac{c_1(Y)^2c_2(Y)}{2} - \frac{c_1(Y)^4}{8}\,\, .
\end{equation}
Substituting (\ref{eq4}) and (\ref{eq5}) into (\ref{eq1}), we obtain 
$$3 = -\frac{1}{720}  \cdot \frac{15}{8}(c_1(Y)^2 -4c_2(Y))c_1(Y)^2 +3\,\, .$$
Hence $(4c_2(Y)-c_1(Y)^2)c_1(Y)^2 = 0$ as claimed.
\end{proof}
As $-c_1(Y) = K_Y$ is nef, a fundamental result of Miyaoka \cite{Mi87} states that the class $3c_2(Y) -c_1(Y)^2$ is pseudo-effective. Hence
$$(3c_2(Y) -c_1(Y)^2) c_1(Y)^2 \ge 0\,\, ,\,\, {\rm i.e.}\,\, ,\,\, 3c_2(Y)c_1(Y)^2 \ge c_1(Y)^4\,\, .$$
If in addition that $Y$ is of general type, then $c_1(Y)^4 > 0$ as $K_Y = -c_1(Y)$ is nef. Therefore
$$c_2(Y)c_1(Y)^2 > 0\,\, .$$
However, then 
$$ (4c_2(Y) - c_1(Y)^2)c_1(Y)^2 > (3c_2(Y) -c_1(Y)^2)c_1(Y)^2 \ge 0\,\, ,$$
a contradiction to Proposition \ref{prop54}. Hence $Y$ is not of general type. 
\end{proof}

As $K_Y$ is nef, it follows that $\nu(Y) \not= 4.$ 
Since $\dim\, Y = 4,$ we { have} $\nu(Y) \le3$ from $\nu(Y) \not= 4.$ Thus $K_Y^4 = 0,$ as $K_Y$ is nef. Hence $q_Y(K_Y) = 0,$ and therefore $K_Y^3 = 0,$ by Lemma \ref{lem52} {(3)}. Hence $\nu(Y) \le 2.$ Assume that $\nu(Y) \le 1.$ Then $K_Y^2 = 0$ in $H^4(Y, \Z).$ Hence $K_Y = 0,$ i.e., $\nu(Y) = 0$ by Lemma \ref{lem52} {(2)}. 

This completes the proof of Theorem \ref{thm4}.

\section{Speculative Exotic algebraic structures on hyperk\"ahler manifolds}

%In this section, we prove Theorem \ref{thm5}. 

Let $n \geq 2$ be an integer and $S$ a smooth projective K3 surface. 

Kodaira \cite{Ko70} found { a minimal compact analytic} elliptic surface $W$ of Kodaira dimension $\kappa(W) = 1,$ which is { o-homotopic} to $S.$ It turns out that $W$ is o-homeomorphic to $S$ by Theorem \ref{thm12} and Remark \ref{rem14}.

%(\cite{Fr82} see also \cite{BHPV04}). Indeed, as $W$ is o-homotopically 
%equivalent to $S,$ both surfaces are simply connected and the second 
%cohomoogy lattice $(H^2(W, \Z), (*, **)_W)$ is isometric to 
%$(H^2(S, \Z), (*, **)_S)$ and they are even and (always) umimodular. 

As $b_1(W) = b_1(S) = 0$ is even, $W$ is deformation equivalent to a smooth { projective} minimal elliptic surface with the same Kodaira dimension due to a fundamental result of Kodaira \cite[Theorem 15.2]{Ko64}, the crucial part of the affirmative answer to the so-called Kodaira problem in dimension $2${ , and the invariance of  the Kodaira dimension for surfaces under deformation (see, e.g., \cite[Section VI, Theorem 8.1]{BHPV04}).} Thus there is a {\it projective} minimal elliptic surface $W$ with $\kappa(W) = 1$ such that $W$ is o-homeomorphic to $S.$ We fix such $W$ from now on. Note however that $W$ is not o-diffeomorphic to $S$ by \cite[Page 495, Corollary 3.4]{FM96}. 

The Hilbert scheme $W^{[n]}$ of $0$-dimensional closed subschemes of length $n$ of $W$ is then a smooth projective $2n$-fold by \cite{Fo68}. We observe
\begin{equation}\label{eq61}
\nu(W^{[n]}) = \kappa(W^{[n]}) = n\,\, .
\end{equation} 
This equality is  seen as follows. First note that $K_W$ is semi-ample as $W$ is a minimal surface of non-negative Kodaira dimension. Since the Hilbert-Chow morphism 
$$W^{[n]} \to W^{(n)} =   {\rm Sym}^n\,(W) $$
is crepant (\cite{Be83}), for any sufficiently divisible large integers $m,$ the pluri-canonical {\it morphism}
$$\Phi_{|mK_W|} : W \to \PP^1$$ 
induces, via the Hilbert-Chow morphism, the  pluri-canonical morphism 
$$\Phi_{|mK_{W^{[n]}}|} : W^{[n]} \to W^{(n)} \to {\rm Sym}^n\,(\PP^1) = \PP^n$$ 
of $W^{[n]}.$  Hence $\nu(W^{[n]}) = \kappa(W^{[n]}) = n.$ This also shows that $K_{W^{[n]}}$ is semi-ample, \cite{Ka85}, hence nef.  

\begin{question} \label{que} 
Let $S$ and $W$ be as above and 
let $$\varphi : S \to W$$ be an o-homeomorphism. Then the product homeomorphism $\varphi^n : S^n \to W^n$ is an o-homeomorphism and $\varphi^n$ induces the o-homeomorphism $\varphi^{(n)} : S^{(n)} \to W^{(n)}.$ However, it is {\it unclear} whether $\varphi^{(n)}$ lifts to an o-homeomorphism $$\varphi^{[n]} : S^{[n]} \to W^{[n]}$$ via the Hilbert-Chow morphisms $S^{[n]} \to S^{(n)}$ and $W^{[n]} \to W^{(n)}.$ If yes, we obtain an exotic complex structure on the hyperk\"ahler manifold $S^{[n]}$ of (numerical) Kodaira dimension $n$. Of course, $
S^{[n]}$ and $W^{[n]}$ might abstractly be o-homeomorphic without $\varphi^{[n]}$ being a o-homeomorphism. 
\end{question}


\begin{thebibliography}{BCHM10} 

\bibitem[An85]{An85} Ando, T., : \textit{On extremal rays of the higher-dimensional varieties}, Invent. Math. {\bf 81}  (1985) 347--357.

\bibitem[AW98]{AW98} Andreatta, M., ; Wisniewski, J. A., : \textit{ On contractions of smooth varieties}, J. Algebraic Geom. {\bf 7} (1998) 253--312. 

\bibitem[BHPV04]{BHPV04} Barth, W. P.; Hulek, K. ; Peters, C. A. M.; Van de Ven, A., : \textit{ Compact complex surfaces. Second edition}, Ergebnisse der Mathematik und ihrer Grenzgebiete. 3. Folge, {\bf 4}, Springer-Verlag, Berlin, 2004.

\bibitem[Be83]{Be83}  Beauville, A., : \textit{ Vari\'et\'es K\"ahleriennes dont la premi\'ere classe de Chern est nulle},  J. Differential Geom.  {\bf 18}  (1983) 755--782 (1984).  

\bibitem[BD85]{BD85} Beauville, A., ; Donagi, R., : \textit{La vari\'et\'e des droites d'une hypersurface cubique de dimension 4}, C. R. Acad. Sci. Paris S\'er. I Math. {\bf 301}  (1985) 703--706.

\bibitem[Br64]{Br64} Brieskorn, E., : \textit{Ein Satz \"uber die komplexen Quadriken} (German), Math. Ann. {\bf 155}  (1964) 184--193.

\bibitem[Ca04]{Ca04} Catanese, F., : \textit{Deformation in the large of some complex manifolds, I}, Ann. Mat. Pura Appl. {\bf 183} (2004) 261--289.

\bibitem[Ca15]{Ca15} Catanese, F., : \textit
{Topological methods in moduli theory},
Bull. Math. Sci. {\bf 5} (2015), 287--449.
 
 

\bibitem[CL97]{CL97} Catanese,F.; LeBrun,C.: \textit{On the scalar curvature of {E}instein manifolds}, Math. Res. Lett. {\bf 4} (1997) 843--854.
 
 \bibitem[CP94]{CP94} Campana,F.; Peternell, Th., : \textit{
Rigidity theorems for primitive Fano 3-folds}, 
Comm. Anal. Geom. {\bf 2} (1994), no. 2, 173--201. 


\bibitem[Fo68]{Fo68} Fogarty, J., : \textit{Algebraic families on an algebraic surface}, Amer. J. Math {\bf 90} (1968) 511--521. 

{ \bibitem[FM96]{FM96} Friedman, R., ; Morgan, J. W., : \textit{Smooth four-manifolds and complex surfaces}, Ergebnisse der Mathematik und ihrer Grenzgebiete. 3. Folge. A Series of Modern Surveys in Mathematics, {\bf 27}, Springer-Verlag, Berlin, 1996.}

\bibitem[Fr82]{Fr82} Freedman, M. H. : \textit{The topology of four-dimensional manifolds}, J. Differential Geom. {\bf 17}  (1982) 357--453.

\bibitem[Fu83]{Fu83} Fujiki, A., : \textit{On primitively symplectic compact K\"ahler V -manifolds of dimension four},  Classification of algebraic and analytic manifolds (Katata, 1982),  71--250, Progr. Math., {\bf 39}, Birkh\"auser Boston, Boston, MA, 1983. 

\bibitem[Fu87]{Fu87} Fujiki, A., : \textit{On the de Rham cohomology group of a compact K\"ahler symplectic manifold},  Algebraic geometry, Sendai, 1985, 
105--165, Adv. Stud. Pure Math., {\bf 10}, North-Holland, Amsterdam, 1987.

\bibitem[Fj90]{Fj90} Fujita, T., : \textit{Classification theories of polarized varieties}, London Mathematical Society Lecture Note Series, {\bf 155}, Cambridge University Press, Cambridge, 1990.

%\bibitem[GH78]{GH78} Griffiths, P., ; Harris, J., : \textit{ Principles of 
%algebraic geometry}, Pure and Applied Mathematics. Wiley-Interscience, 
%New York, 1978.

\bibitem[GHJ03]{GHJ03} Gross, M.; Huybrechts, D.; Joyce, D., : \textit{ Calabi-Yau manifolds and related geometries}, Lectures from the Summer School held in Nordfjordeid, June 2001. Universitext. Springer-Verlag, Berlin, 2003. 

%{ \bibitem[SGA1]{SGA1} Grothendieck, A. et al, : \textit{Rev\^etements \'etales et groupe fondamental}, Letcure Notes Math. {\bf 224}, Springer, Berlin 1971. (Place has been changed fine?)} 

\bibitem[Gu01]{Gu01} Guan, D., : \textit{On the Betti numbers of irreducible compact hyperkahler manifolds of complex dimension four}, 
Math. Res. Lett.  {\bf 8} (2001) 663--669. 

\bibitem[HK57]{HK57} Hirzebruch, F.,  Kodaira, K. : \textit{On the complex projective spaces}, J. Math. Pures Appl. {\bf 36} (1957) 201--216. 

%\bibitem[Is77]{Is77} Iskovskih, V. A., : \textit{ Fano threefolds. I.}, 
%Izv. Akad. Nauk SSSR Ser. Mat. {\bf 41}  (1977) 516--562.

%\bibitem[Is78]{Is78} Iskovskih, V. A., : \textit{ Fano threefolds. II.}, 
%Izv. Akad. Nauk SSSR Ser. Mat. {\bf 42}  (1978) 506--549.

\bibitem[Ka84]{Ka84} Kawamata, Y., : \textit{ The cone of curves of algebraic varieties},  Ann. of Math. {\bf 119} (1984) 603--633.

\bibitem[Ka85]{Ka85} Kawamata, Y., : \textit{Pluricanonical systems on minimal algebraic varieties}, Invent. Math. {\bf 79}  (1985) 567--588.

\bibitem[Ka89]{Ka89}  Kawamata, Y., : \textit{ Small contractions of four-dimensional algebraic manifolds}, Math. Ann.  {\bf 284}  (1989) 595--600. 

\bibitem[KMM87]{KMM87} Kawamata, Y.; Matsuda, K.; Matsuki, K., : \textit{ Introduction to the minimal model problem},  Algebraic geometry, Sendai, 1985,  283--360, Adv. Stud. Pure Math., {\bf 10}, North-Holland, Amsterdam, 1987.

\bibitem[KO73]{KO73} Kobayashi, S., ; Ochiai, T., : \textit{ Characterizations of complex projective spaces and hyperquadrics}, J. Math. Kyoto Univ. {\bf 13}  (1973) 31--47. 

\bibitem[Ko63]{Ko63} Kodaira, K., : \textit{ On compact analytic surfaces. II, III.}, Ann. of Math. {\bf 77} (1963), 563--626; ibid. {\bf 78} (1963) 1--40.

\bibitem[Ko64]{Ko64} Kodaira, K., : \textit{ On the structure of compact complex analytic surfaces. I.},  Amer. J. Math.  {\bf 86}  (1964) 751--798.

\bibitem[Ko70]{Ko70} Kodaira, K., \textit{ On homotopy K3 surfaces}, 1970  Essays on Topology and Related Topics (M\'emoires d\'edi\'es \'a Georges de Rham) 
58--69 Springer, New York.  

\bibitem[KM98]{KM98} Koll\'ar, J., ; Mori, S., : \textit{Birational geometry of algebraic varieties. With the collaboration of C. H. Clemens and A. Corti}, Translated from the 1998 Japanese original. Cambridge Tracts in Mathematics, 134. Cambridge University Press, Cambridge, 1998.  

\bibitem[Kol91]{Kol91} Koll\'ar, J., : \textit{Flips, flops, minimal models, etc},  Surveys in differential geometry (Cambridge, MA, 1990),  113--199, Lehigh Univ., Bethlehem, PA, (1991). 

%\bibitem[Kol95]{Kol95} Koll\'ar, J., : \textit{ Shafarevich maps and 
%automorphic forms}, M. B. Porter Lectures. Princeton University Press, 
%Princeton, NJ, 1995.

\bibitem[Kol96]{Kol96} Koll\'ar, J., : \textit{Rational curves on algebraic varieties}, Ergebnisse der Mathematik und ihrer Grenzgebiete. 3. Folge. A Series of Modern Surveys in Mathematics, {\bf 32}, Springer-Verlag, Berlin, 1996.

\bibitem[Ku10]{Ku10} Kuznetsov, A., : \textit{Derived categories of cubic fourfolds},  Cohomological and geometric approaches to rationality problems, 219--243, Progr. Math., {\bf 282} Birkh\"auser Boston, Inc., Boston, MA, 2010. 

%{ \bibitem[LP07]{LP07} Lee, Y., ; Park, J., : \textit{A simply connected surface of general type with $p_g = 0$ and $K^2 = 2$}, Invent. Math. {\bf 170} (2007) 483--505.}  

%\bibitem[Me99]{Me99} Mella, M., : \textit{ Existence of good divisors on 
%Mukai varieties}, J. Algebraic Geom.  {\bf 8}  (1999) 197--206. 


\bibitem[LW90]{LW90} 
Libgober, A., Wood, J. W.
\textit{Uniqueness of the complex structure on K\"ahler manifolds of
              certain homotopy types},
J. Differential Geom. {\bf 32} (1990), 139--154.
 
 	






%\bibitem[Mi65]{Mi65} Milnor, J. W., : \textit{Lectures on the $h$-cobordism theorem. Notes by L. Siebenmann and J. Sondow}, Princeton University Press, Princeton, N.J. 1965 

\bibitem[MS74]{MS74} Milnor, J. W.; Stasheff, J. D., : \textit{Characteristic classes}, Annals of Mathematics Studies, {\bf 76},  Princeton University Press, Princeton, N. J.; University of Tokyo Press, Tokyo, 1974.

\bibitem[Mi87]{Mi87} Miyaoka, Y., : \textit{The Chern classes and Kodaira dimension of a minimal variety},  Algebraic geometry, Sendai, 1985,  449--476, Adv. Stud. Pure Math., {\bf 10}, North-Holland, Amsterdam, 1987.

%\bibitem[MM81]{MM81} Mori, S., ; Mukai, S., : \textit{Classification of Fano 
%$3$-folds with $B_2 \ge 2$}, Manuscripta Math.  {\bf 36}  (1981/82) 147--162.

%\bibitem[MM03]{MM03} Mori, S., ; Mukai, S., : \textit{ Erratum: "Classification%of Fano $3$-folds with $B_2 \ge 2$}, Manuscripta Math.  {\bf 110}  (2003) 407. 

%\bibitem[Mo84]{Mo84} Morrison, D. R., : \textit{ On K3 surfaces with large 
%Picard number}, Invent. Math. {\bf 75}  (1984) 105--121. 


\bibitem[Na96]{Na96} Nakamura, I., : \textit{Moishezon threefolds homeomorphic to a cubic hypersurface in $\PP^4$}, J. Algebraic Geom.  {\bf 5} (1996) 537--569.
\bibitem[No65]{No65} Novikov, S. P., : \textit{Topological invariance of rational classes of Pontrjagin},  Dokl. Akad. Nauk SSSR  {\bf 163} (1965) 298--300. 

%\bibitem[OS11]{OS11} Oguiso, K., ; Schr\"oer, S., : \textit{Enriques manifolds}, J. Reine Angew. Math.  {\bf 661}  (2011) 215--235. 

\bibitem[PZ16]{PZ16} Prokhorov,Y.; Zaidenberg, M.: \textit{Examples of cylindrical Fano fourfolds}. Eur. J. Math. {\bf 2} (2016) 262-282

\bibitem[Ra06]{Ra06} 
Rasdeaconu, R., \textit{The Kodaira dimension of diffeomorphic K\"ahler threefolds},
Proc. Amer. Math. Soc. {\bf 134} (2006), 3543--3553.



%\bibitem[Se73]{Se73} Serre, J.-P., : \textit{A course in arithmetic}, 
%Graduate Texts in Mathematics, {\bf 7} Springer-Verlag, New York-
%Heidelberg, 1973.

\bibitem[Ta99]{Ta99} Takagi, H., : \textit{ Classification of extremal contractions from smooth fourfolds of (3,1)-type}, Proc. Amer. Math. Soc. {\bf 127}  (1999) 315--321.

\bibitem[Ve96]{Ve96} Verbitsky, M., : \textit{Cohomology of compact hyper-K\"ahler manifolds and its applications}, Geom. Funct. Anal. {\bf 6} (1996) 601--611.

\bibitem[Vo86]{Vo86} Voisin, C., : \textit{Th\'eor\`eme de Torelli pour les cubiques de $\PP^5$}, Invent. Math. {\bf 86}  (1986) 577--601.

\bibitem[Vo02]{Vo02} Voisin, C., : \textit{Hodge theory and complex algebraic geometry. I.},  Cambridge Studies in Advanced Mathematics, {\bf 76}, Cambridge University Press, Cambridge, 2002.

%\bibitem[Wa60]{Wa60} Wall, C. T. C., : \textit{Determination of the cobordism ring.}, Ann. of Math. {\bf 72} (1960) 292--311. 

\bibitem[Ya77]{Ya77}  Yau, S.- T., : \textit{Calabi's conjecture and some new results in algebraic geometry}, Proc. Nat. Acad. Sci. U.S.A. {\bf 74}  (1977) 
1798--1799.

\end{thebibliography}
\end{document}